\documentclass[12pt]{amsart}
\usepackage[nobysame,abbrev,alphabetic]{amsrefs}
\usepackage{amssymb}
\usepackage[all]{xy}

\DeclareMathOperator{\Coker}{Coker}
\DeclareMathOperator{\Gal}{Gal}
\DeclareMathOperator{\Hom}{Hom}

\DeclareMathOperator{\Img}{Im}

\DeclareMathOperator{\Ker}{Ker}

\DeclareMathOperator{\res}{res}

\DeclareMathOperator{\trg}{trg}

\begin{document}

\newtheorem{thm}{Theorem}[section]
\newtheorem{cor}[thm]{Corollary}
\newtheorem{lem}[thm]{Lemma}
\newtheorem{prop}[thm]{Proposition}
\newtheorem{defin}[thm]{Definition}
\newtheorem{exam}[thm]{Example}
\newtheorem{examples}[thm]{Examples}
\newtheorem{rem}[thm]{Remark}
\newtheorem{case}{\sl Case}
\newtheorem{claim}{Claim}
\newtheorem{prt}{Part}
\newtheorem*{mainthm}{Main Theorem}
\newtheorem*{thmA}{Theorem A}
\newtheorem*{thmB}{Theorem B}
\newtheorem*{thmC}{Theorem C}
\newtheorem*{thmD}{Theorem D}
\newtheorem{question}[thm]{Question}
\newtheorem*{notation}{Notation}
\swapnumbers
\newtheorem{rems}[thm]{Remarks}
\newtheorem*{acknowledgment}{Acknowledgment}
\newtheorem{questions}[thm]{Questions}
\numberwithin{equation}{section}

\newcommand{\ab}{\mathrm{ab}}
\newcommand{\Bock}{\mathrm{Bock}}
\newcommand{\dec}{\mathrm{dec}}
\newcommand{\dirlim}{\varinjlim}
\newcommand{\discup}{\ \ensuremath{\mathaccent\cdot\cup}}
\newcommand{\divis}{\mathrm{div}}
\newcommand{\Galp}{\underline{\mathrm{Gal}}_p}
\newcommand{\nek}{,\ldots,}
\newcommand{\inv}{^{-1}}
\newcommand{\isom}{\cong}
\newcommand{\pr}{\mathrm{pr}}
\newcommand{\PrGr}{\underline{\mathrm{PrGr}}}
\newcommand{\sep}{\mathrm{sep}}
\newcommand{\tagg}{^{''}}
\newcommand{\tensor}{\otimes}
\newcommand{\alp}{\alpha}
\newcommand{\gam}{\gamma}
\newcommand{\Gam}{\Gamma}
\newcommand{\del}{\delta}
\newcommand{\Del}{\Delta}
\newcommand{\eps}{\epsilon}
\newcommand{\lam}{\lambda}
\newcommand{\Lam}{\Lambda}
\newcommand{\sig}{\sigma}
\newcommand{\Sig}{\Sigma}
\newcommand{\bfA}{\mathbf{A}}
\newcommand{\bfB}{\mathbf{B}}
\newcommand{\bfC}{\mathbf{C}}
\newcommand{\bfF}{\mathbf{F}}
\newcommand{\bfP}{\mathbf{P}}
\newcommand{\bfQ}{\mathbf{Q}}
\newcommand{\bfR}{\mathbf{R}}
\newcommand{\bfS}{\mathbf{S}}
\newcommand{\bfT}{\mathbf{T}}
\newcommand{\bfZ}{\mathbf{Z}}
\newcommand{\dbA}{\mathbb{A}}
\newcommand{\dbC}{\mathbb{C}}
\newcommand{\dbF}{\mathbb{F}}
\newcommand{\dbN}{\mathbb{N}}
\newcommand{\dbQ}{\mathbb{Q}}
\newcommand{\dbR}{\mathbb{R}}
\newcommand{\dbU}{\mathbb{U}}
\newcommand{\dbZ}{\mathbb{Z}}
\newcommand{\grf}{\mathfrak{f}}
\newcommand{\gra}{\mathfrak{a}}
\newcommand{\grm}{\mathfrak{m}}
\newcommand{\grp}{\mathfrak{p}}
\newcommand{\grq}{\mathfrak{q}}
\newcommand{\calA}{\mathcal{A}}
\newcommand{\calB}{\mathcal{B}}
\newcommand{\calC}{\mathcal{C}}
\newcommand{\calE}{\mathcal{E}}
\newcommand{\calG}{\mathcal{G}}
\newcommand{\calH}{\mathcal{H}}
\newcommand{\calK}{\mathcal{K}}
\newcommand{\calL}{\mathcal{L}}
\newcommand{\calW}{\mathcal{W}}
\newcommand{\calV}{\mathcal{V}}

\author{Ido Efrat}
\address{Mathematics Department\\
Ben-Gurion University of the Negev\\
P.O.\ Box 653, Be'er-Sheva 84105\\
Israel}
\email{efrat@math.bgu.ac.il}

\author{J\'an Min\'a\v c}
\address{Mathematics Department\\
University of Western Ontario\\
London\\
Ontario\\
Canada N6A 5B7}
\email{minac@uwo.ca}

\thanks{
The first author was supported by the Israel Science Foundation (grants No.\ 23/09 and 152/13).
The second author was supported in part by National Sciences and Engineering Council of Canada grant R0370A01.}

\begin{abstract}
Let $q=p^s$ be a prime power, $F$ a field containing a root of unity of order $q$, and $G_F$ its absolute Galois group.
We determine a new canonical quotient $\Gal(F_{(3)}/F)$ of $G_F$ which encodes the full mod-$q$ cohomology ring $H^*(G_F,\dbZ/q)$
and is minimal with respect to this property.
We prove some fundamental structure theorems related to these quotients.
In particular, it is shown that when $q=p$ is an odd prime, $F_{(3)}$ is the compositum of all Galois extensions $E$ of $F$
such that $\Gal(E/F)$ is isomorphic to $\{1\}$, $\dbZ/p$ or to the nonabelian group $H_{p^3}$
of order $p^3$ and exponent $p$.
\end{abstract}

\title{Galois groups and cohomological functors}

\maketitle

\section{Introduction}
Let $p$ be a fixed prime number and $q=p^s$ a fixed $p$-power, where $s\geq1$.
For a profinite group $G$, let  $H^*(G)=\bigoplus_{i=0}^\infty H^i(G,\dbZ/q)$ be the cohomology (graded)
ring with the trivial action of $G$ on $\dbZ/q$.
We will be mostly interested in the case where $G=G_F$ is the absolute Galois group of a field $F$
which contains a root of unity of order $q$.
The ring $H^*(G_F)$, and even its degree $\leq2$ part, is known to encode important arithmetical information on $F$
(see below).
In this paper we ask how much Galois-theoretic information is needed to compute $H^*(G_F)$.
More specifically, we characterize the minimal quotient of $G_F$ which determines it.
In \cite{CheboluEfratMinac} it was shown that  $H^*(G_F)$
is determined by a quite small quotient $G_F^{[3]}=G_F/(G_F)^{(3)}$ of $G_F$.
Here $G^{(i)}=G^{(i,q)}$, $i=1,2,3\nek$ is the \textsl{descending $q$-central sequence} of a given profinite $G$
(see \S\ref{section on cohomology}).
Namely, the inflation map $\inf\colon H^*(G_F^{[3]})_\dec\to H^*(G_F)$ is an isomorphism,
where $H^*(G_F^{[3]})_\dec$ denotes the subring of $H^*(G_F^{[3]})$ generated by degree $1$ elements.
In the present paper we first find an even smaller quotient $(G_F)_{[3]}=G_F/(G_F)_{(3)}$ of $G_F$ with the same property.
Here for a profinite group $G$ we set $G_{(3)}=G^q[G^{(2)},G]$ if $p>2$, and
$G_{(3)}=G^{2q}[G^{(2)},G]$ if $p=2$  (see \S\ref{section on cohomology}).
Thus, when $q=p$, $G_{(3)}=G^{(3)}$ is the third term in the Zassenhaus $p$-filtration of $G$.
Moreover, we prove that this quotient is  \textsl{minimal} with respect to this property:

\begin{thmA}
Let $N$ be a closed normal subgroup of $G_F$.
Then the inflation map $H^*(G_F/N)_\dec\to H^*(G_F)$ is an isomorphism if and only if $N\leq (G_F)_{(3)}$.
In particular,  $(G_F)_{[3]}$ determines $H^*(G_F)$.
\end{thmA}
Theorem A is proved in \S\ref{section on quotients that determine cohomology}.
Conversely, in \S\ref{section quotients determined by cohomology} we prove:

\begin{thmB}
The degree $\leq2$ part of  $H^*(G_F)$ determines $(G_F)_{[3]}$.
\end{thmB}

We also prove the following strengthening of \cite{CheboluEfratMinac}*{Th.\ C, D}.
Denote the maximal pro-$p$ Galois group of $F$ by $G_F(p)$.

\begin{thmC}
Let $F_1,F_2$ be fields containing a $q$th root of unity.
Let $\pi\colon G_{F_1}(p)\to G_{F_2}(p)$ be a continuous homomorphism and
$\pi^*\colon H^*(G_{F_2}(p))\to H^*(G_{F_1}(p))$, $\pi_{[3]}\colon (G_{F_1})_{[3]}\to (G_{F_2})_{[3]}$
the induced maps.
Then $\pi$ is an isomorphism if and only if $\pi_{[3]}$ is an isomorphism if and only if $\pi^*$ is an isomorphism.
\end{thmC}
The proof of Theorem C is given in \S\ref{section on isomorphisms}.
Theorem C is in the spirit of the ``anabelian phenomena" in Galois theory \cite{NeukirchSchmidtWingberg}*{Ch.\ XII, \SS2--3}, as it shows that certain isomorphisms of (small) Galois groups already imply that the much richer structures of the maximal pro-$p$ Galois groups are isomorphic.

Regarding the structure of the quotient $(G_F)_{[3]}$, we prove that it is the Galois group of the
compositum of all Galois extensions of $F$ with certain specific Galois groups:

\begin{thmD}
Assume that $q=p\neq2$ is prime and let $F$ be as above.
Then $(G_F)_{(3)}$ is the intersection of all normal open subgroups $N$ of $G_F$
such that $G_F/N$ is isomorphic to $\dbZ/p$ or $H_{p^3}$.
\end{thmD}
Here $H_{p^3}$ is the non-abelian group of odd order $p^3$ and exponent $p$ (the \textsl{Heisenberg group}):
\[
H_{p^3}=\bigl\langle  r,s,t\ \bigm|\  r^p=s^p=t^p=[r,t]=[s,t]=1,\ [r,s]=t \bigr\rangle.
\]
In the remaining case $q=p=2$ the group $(G_F)_{(3)}$ is known to be the
intersection of all normal open subgroups $N$ of $G_F$
such that $G/N$ is isomorphic to either $\dbZ/2$, $\dbZ/4$, or the dihedral group $D_4$ of order $8$
(Villegas \cite{Villegas88}, Min\'a\v c--Spira \cite{MinacSpira96}*{Cor.\ 2.18};
see also \cite{EfratMinac11}*{Cor.\ 11.3 and Prop.\ 3.2}).
Moreover, $\dbZ/2$ can be omitted from this list unless $F$ is Euclidean \cite{EfratMinac11}*{Cor.\ 11.4}.
The proof of Theorem D is given in Example \ref{tyty}(1).

Theorem A gives a new restriction on the structure of maximal pro-$p$ Galois groups of fields as above.
Indeed, it implies that if the defining relations in such a group $G=G_F(p)$ are changed within $G_{(3)}$ and the cohomology ring is changed, then the resulting group may not be realized as Galois group in this way
(see Corollary \ref{only one Galois group} for a precise statement).
In particular, Theorem A directly implies the classical Artin--Schreier theorem, asserting that
elements of absolute Galois groups can have only order $1$, $2$, or $\infty$ (Remark \ref{examples}(2)).

Our proofs of Theorems A, C, D are based on the bijectivity of the Galois symbol homomorphism
$K^M_*(F)/qK^M_*(F)\to H^*(G_F)$, proved by Rost and Voevodsky (with a patch by Weibel),
where $K^M_*(F)$ is the Milnor $K$-ring of $F$;
see  \cite{Voevodsky03a}, \cite{Voevodsky11}, \cite{Weibel09}, \cite{Weibel08}.
The bijectivity of the Galois symbol in degree $2$ was proved earlier by Merkurjev and Suslin \cite{MerkurjevSuslin82}.
Specifically, the proofs of Theorems A and the second equivalence in Theorem C use the bijectivity of this map,
the proof of the first equivalence in Theorem C uses its bijectivity in degree $2$,
and that of Theorem D uses only its injectivity in degree $2$.

Our approach is purely group-theoretic, and is based on a fundamental notion of duality between a pair $(T,T_0)$
of normal subgroups of a profinite group $G$
and a subgroup $A$ of the second cohomology $H^2(G/T)$ (see \S\ref{section on duality}).
When applied to $T=G^{(2)},T_0=G_{(3)}$ and $A=H^2(G^{[2]})_\dec$, with $G=G_F$ as above, it leads to Theorems A--D.
Moreover, it can also be applied to other choices of $T,T_0,A$ as well to yield analogous results.
Notably, taking $T=G^{(2)},T_0=G^{(3)}$ and $A=H^2(G^{[2]})$, we strengthen the main results of \cite{CheboluEfratMinac}.
As a third example we may take $T=G^{(2)},T_0=G^{q^2}[G,G]$ and $A=\Img(\beta_{G^{[2]}})$,
where $\beta_{G^{[2]}}\colon H^1(G^{[2]})\to H^2(G^{[2]})$ is the Bockstein map (see \S\ref{section on cohomology}).

The Galois group $(G_F)_{[3]}$ encodes important arithmetical information on $F$.
For instance, when $q=2$ it was shown in \cite{MinacSpira90}, \cite{MinacSpira96}
that $(G_F)_{[3]}$ and the Kummer element of $-1$
encode the orderings on $F$ as well as the Witt ring of quadratic forms over $F$.
Moreover, in \cite{EfratMinac12} we use Theorem D to prove that  $(G_F)_{[3]}$ also encodes
a class of valuations which are important in the pro-$p$ context,
namely, (Krull) valuations $v$ whose value group is not divisible by $p$ and
whose $1$-units are $p$-powers (this is a weak form of Hensel's lemma).
Specifically, under a finiteness assumption, and assuming the existence of
a root of unity of order $4$ in $F$ when $p=2$,
there exists such a valuation if and only if the center $Z((G_F)_{[3]})$ has a nontrivial image in $G_F^{[2]}$.

From a more general perspective, recovering arithmetical information on a field from the profinite group-structure of its absolute Galois group, and even of smaller canonical Galois groups, is the main theme of the birational anabelian geometry.
See for instance the recent informative review \cite{BogomolovTschinkel12}, as well as \cite{Pop12}, \cite{Topaz12} and the introduction of \cite{MinacSwallowTopaz14}.
Typically such results are limited to various geometric situations.
The fact that considerable crucial arithmetical information on valuations and orderings can be recovered even from the small quotient $(G_F)_{[3]}$ (with the only assumption about having a root of unity) is a surprising consequence of this work.

The first version of this work was posted to the archive arXiv in 2011.
Our goal was to develop such techniques which would be adaptable for various other higher $p$-descending filtrations of absolute Galois groups.
Since then there have indeed been various related subsequent works on the $p$-Zassenhaus filtration $G_{(n)}$ of a profinite group $G$,
Massey products in Galois cohomology (which generalize the cup product), and duality principles between them of this nature.
In \cite{Efrat14a}*{Th.\ B} it was shown that when $G$ is a free profinite group, then the subgroup $A$ of $H^2(G/G_{(n)})$ generated by all $n$-fold Massey products is dual to $(G_{(n)},G_{(n+1)})$, in the sense of \S3.
Further, let $\dbU_n(\dbF_p)$ be the group of all $n\times n$ upper-triangular unipotent matrices over $\dbF_p$.
It was observed  in \cite{Efrat14a} that Theorem D above can be rephrased as saying that for $G=G_F$ and $q=p$ a prime number, one has $G_{(3)}=\bigcap\Ker(\rho)$, where the intersection is over all continuous homomorphisms $\rho\colon G\to\dbU_3(\dbF_p)$.
It was shown that for every $n\geq3$ profinite groups $G$ satisfying a certain cohomological condition on $(n-1)$-fold Massey products, one has
\begin{equation}
\label{unipotent intersection}
G_{(n)}=\bigcap\{\Ker(\rho)\ |\ \rho\colon G\to \dbU_n(\dbF_p)\}.
\end{equation}
In particular, this holds when $G$ is free pro-$p$ (see also \cite{Efrat14b} and \cite{MinacTan15b}*{Th.\ 2.6} for direct alternative proofs).
Then, in \cite{MinacTan15a}*{\S8}  and \cite{MinacTan15b},  it was conjectured that the equality (\ref{unipotent intersection}) holds when $G=G_F$ is an absolute Galois group of a field $F$ containing a root of unity of order $p$ (the {\sl ``kernel $n$-unipotent conjecture"}).
It was shown in \cite{MinacTan15b} that (\ref{unipotent intersection}) holds for profinite groups related to local fields: odd rigid groups, and Demu\v skin groups.

We thank the referee for his/her thoughtful comments, which helped us to improve the paper at several points.

\section{Preliminaries}
\label{section on cohomology}
\subsection*{A) Profinite groups}
We work in the category of profinite groups.
Thus all homomorphisms of profinite groups will be tacitly assumed to be continuous and all subgroups will
be closed.
The \textbf{descending $q$-central filtration} $G^{(i)}=G^{(i,q)}$ of a profinite group $G$ is defined inductively by
\[
G^{(1)}=G,\quad  G^{(i)}=(G^{(i-1)})^q[G^{(i-1)},G],\ i\geq2.
\]
Thus $G^{(i)}$ is generated by all $q$th powers of elements of $G^{(i-1)}$ and all commutators $[h,g]=h\inv g\inv hg$,
where $h\in G^{(i-1)}$, $g\in G$.
Set $G^{[i]}=G/G^{(i)}$.
We also define
\begin{equation}
\label{delta}
\del=\begin{cases}
1 & p>2\\
2& p=2,
\end{cases}
\qquad
G_{(3)}=G^{\del q}[G^{(2)},G],
\qquad
G_{[3]}=G/G_{(3)}.
\end{equation}

\begin{rems}
\label{remark on the def of G3}
\rm
(1)\quad
As $[h,g]=h^{-2}(hg^{-1})^2g^2$, we have $[G,G]\leq G^2$.
Therefore, when $q=2$ we have $G^{(2)}=G^2$, so $G^{(3)}=G^4[G^{(2)},G]=G_{(3)}$.

\medskip

(2)\quad
One has $G^{(3)}\leq G_{(3)}$.
Indeed, this is immediate when $p>2$.
When $p=2$ we have $g^{q^2}=(g^{q/2})^{2q}$ for every $g\in G$,
so $(G^q)^q\leq G^{2q}[G^q,G]$.
Further, one has the identities (see \cite{Labute66}*{Prop.\ 5} and its proof)
\[
(g_1g_2)^q\equiv g_1^qg_2^q[g_2,g_1]^{\left({q\atop2}\right)}, \quad
[g_1g_2,h]\equiv[g_1,h][g_2,h] \pmod{[[G,G],G]}
\]
It follows that $[g_1^2,g_2^2]^{\left({q\atop2}\right)}\in G^{2q}[[G,G],G]$ for any $g_1,g_2\in G$.
By (1), $[G,G]^q\leq(G^2)^q\leq G^{2q}[[G,G],G]$.
Consequently, $(G^{(2)})^q\leq G^{2q}[G^{(2)},G]=G_{(3)}$, whence our claim.
\end{rems}

\subsection*{B) Profinite cohomology}
We refer to \cite{NeukirchSchmidtWingberg}, \cite{RibesZalesskii}, or \cite{SerreCG}
for basic notions and facts in profinite cohomology.
In particular, given a profinite group $G$, let $H^i(G)=H^i(G,\dbZ/q)$ be the $i$th profinite cohomology group of $G$
with respect to its trivial action on $\dbZ/q$.
Thus $H^1(G)=\Hom_{\textrm{cont}}(G,\dbZ/q)$.
Let $H^*(G)=\bigoplus_{r=0}^\infty H^r(G)$ be the cohomology ring with the cup product $\cup$.
Given a homomorphism $\pi\colon G_1\to G_2$ of profinite groups,
let $\pi^*_r\colon H^r(G_2)\to H^r(G_1)$ and $\pi^*\colon H^*(G_2)\to H^*(G_1)$ be the induced homomorphisms.
We write $\res,\inf,\trg$ for the restriction, inflation, and transgression maps, respectively.
For a normal subgroup $N$ of $G$, there is a canonical action of $G$ on $H^i(N)$.
When $i=1$ it is given by $\psi\mapsto\psi^g$, where $\psi^g(n)=\psi(g\inv ng)$
for $g\in G$ and $n\in N$.
We denote the group of all $G$-invariant elements of $H^i(N)$ by $H^i(N)^G$.

For each $r\geq0$, the cup product induces a homomorphism $H^1(G)^{\tensor r}\to H^r(G)$ of $\dbZ/q$-modules.
Let $H^r(G)_\dec$ be its image (where $H^0(G)_\dec=\dbZ/q$),
 and let $H^*(G)_\dec=\bigoplus_{r=0}^\infty H^r(G)_\dec$ be
the \textbf{decomposable cohomology ring} with the cup product.

Following \cite{CheboluEfratMinac}*{\S3}, set $\widehat{H^*(G)}=\bigoplus_{r=0}^\infty H^1(G)^{\tensor r}/C_r$,
where $C_r$ is the subgroup of $H^1(G)^{\tensor r}$ generated by all elements
$\psi_1\tensor\cdots\tensor \psi_r$ such
that $\psi_i\cup\psi_j=0$ for some $i<j$ .
It is a graded ring and there is a canonical graded ring epimorphism $\omega_G\colon \widehat{H^*(G)}\to H^*(G)_\dec$.
The ring $\widehat{H^*(G)}$ and the map $\omega_G$ are functorial in the natural sense.
We call $H^*(G)$ \textbf{quadratic} (resp., \textbf{$r$-quadratic}) if $\omega_G$ is an isomorphism
(resp., in degree $r$).

\begin{lem}
\label{ker of inflation}
When $H^*(G)$ is $2$-quadratic, the kernel of \/
$\inf_G\colon H^2(G^{[2]})_\dec\to H^2(G)$ is generated by elements of the form $\psi\cup\psi'$,
with $\psi,\psi'\in H^1(G^{[2]})$ such that $\inf_G(\psi\cup\psi')=0$.
\end{lem}
\begin{proof}
This follows from the commutative diagram
\[
\xymatrix{
H^1(G^{[2]})^{\tensor2}\ar[r]^{\inf_G}_{\sim}\ar@{->>}[d]_{\cup} & H^1(G)^{\tensor2}\ar[r] &
H^1(G)^{\tensor2}/C_2\ar@{>->}[d]^{\cup}  \\
H^2(G^{[2]})_\dec\ar[rr]^{\inf} && H^2(G)_\dec
}
\]
and the definition of $C_2$.
\end{proof}

The \textbf{Bockstein map} $\beta_G\colon H^1(G)\to H^2(G)$ is the connecting homomorphism arising from
the short exact sequence of trivial $G$-modules
\[
0\to\dbZ/q\to\dbZ/q^2\to\dbZ/q\to0.
\]
It is functorial in the natural way.
The following lemma was proved in \cite{EfratMinac11}*{Cor.\ 2.11} for $I$ finite, and follows in general by a limit argument.
See also Proposition \ref{cup products} for an alternative proof.

\begin{lem}
\label{H2 for elementary abelian groups}
If $G=(\dbZ/q)^I$ for some $I$, then $H^2(G)=\Img(\beta_G)+H^2(G)_\dec$.
\end{lem}

For $r\geq0$ let $B_r$ be the subgroup of $H^1(G)^{\tensor r}$ generated by all elements
$\psi_1\tensor\cdots\tensor \psi_r$ such that $\beta_G(\psi_i)=0$ for some $i$.
We define a graded ring $H^*(G)_\Bock=\bigoplus_{r=0}^\infty H^1(G)^{\tensor r}/B_r$ with the tensor product.

\subsection*{C) Galois cohomology}
The following theorem collects the cohomological properties of an absolute Galois group which are needed
for this paper.

\begin{thm}
\label{absolute Galois groups}
Let $F$ be a field containing a root of unity $\zeta_q$ of order $q$
and let $G_F$ be its absolute Galois group.
Then:
\begin{enumerate}
\item[(i)]
$G_F^{[2]}\isom(\dbZ/q)^I$;
\item[(ii)]
there exists $\xi\in H^1(G_F)$ with $\beta_G(\psi)=\psi\cup\xi$ for every $\psi\in H^1(G_F)$;
\item[(iii)]
$H^*(G_F)=H^*(G_F)_\dec$;
\item[(iv)]
$H^*(G_F)$ is quadratic.
\end{enumerate}
\end{thm}
\begin{proof}
We identify the group $\mu_p$ of $q$th roots of unity in $F$ with $\dbZ/q$, where $\zeta_q$ corresponds to $1$ mod $q$.
The obvious map
$F^\times/(F^\times)^q\to F^\times/(F^\times)^p$ is surjective, and by the Kummer isomorphism,
so is the functorial map $H^1(G_F)=H^1(G_F,\dbZ/q)\to H^1(G_F,\dbZ/p)$.
This implies (i).
In (ii) one takes $\xi$ to be the Kummer element corresponding to $\zeta_q$ (see \cite{EfratMinac11}*{Prop.\ 3.2}).
Condition (iii) (resp., (iv)) follows from the surjectivity (resp., injectivity) of the Galois symbol
$K^M_*(F)/q\to H^*(G_F)$, as proved by Rost, Voevodsky, with a patch by Weibel (see \cite{CheboluEfratMinac}*{\S8}).
\end{proof}

\begin{rem}
\rm
Many of our results do not require the full strength of conditions (iii) and (iv)
of Theorem \ref{absolute Galois groups}, but only their validity in degrees
$1$ and $2$.
These weaker facts follow from the Kummer isomorphism and the Merkurjev--Suslin theorem
(\cite{MerkurjevSuslin82}, \cite{GilleSzamuely}), respectively.
\end{rem}

\begin{rem}
\label{GF(p)}
\rm
By \cite{CheboluEfratMinac}*{Remark 8.2}, $\inf\colon H^*(G_F(p))\to H^*(G_F)$ is an isomorphism.
Therefore (i)--(iv) hold also for $G_F(p)$.
\end{rem}

\section{Duality}
\label{section on duality}
Let $G$ be a profinite group and let $T,T_0$ be normal subgroups of $G$ such that $T^q[T,G]\leq T_0\leq T\leq G^{(2)}$.
Let
\[
K=\Ker\bigl(H^1(T)^G\xrightarrow{\res} H^1(T_0)\bigr),\quad
K'=\Ker\bigl(H^2(G/T)\xrightarrow{\inf} H^2(G/T_0)\bigr).
\]
Note that if $T_0=T^q[T,G]$, then $K=H^1(T)^G$.
As $T,T_0\leq G^{(2)}$, the inflations $H^1(G/T)\to H^1(G)$, $H^1(G/T_0)\to H^1(G)$ are isomorphisms.
The functoriality of the 5-term sequence \cite{NeukirchSchmidtWingberg}*{pp.\ 78--79}
gives a commutative diagram with exact rows
\begin{equation}
\label{5 term sequence}
\xymatrix{
0\ar[r] & H^1(T)^G\ar[d]_{\res} \ar[r]^{\trg} & H^2(G/T)\ar[d]^{\inf} \ar[r]^{\ \inf} & H^2(G)\ar@{=}[d] \\
0\ar[r] & H^1(T_0)^G \ar[r]^{\trg} & H^2(G/T_0) \ar[r]^{\ \inf} & H^2(G). \\
}
\end{equation}
By the snake lemma, $K,K'$ are therefore isomorphic via transgression.

There is a perfect duality
\begin{equation}
\label{pairing1}
T/T^q[T,G]\times H^1(T)^G\to\dbZ/q, \quad (\bar\sig,\psi)\mapsto\psi(\sig)
\end{equation}
which is functorial in the natural sense \cite{EfratMinac11}*{Cor.\ 2.2}.
It induces a (functorial) perfect duality
\begin{equation}
\label{pairing2}
T/T^q[T,G]\times \trg(H^1(T)^G)\to\dbZ/q, \quad \langle\bar\sig,\varphi\rangle\mapsto(\trg\inv(\varphi))(\sig).
\end{equation}

\begin{prop}
\label{generalized basic duality}
(\ref{pairing1}) and (\ref{pairing2}) induce perfect pairings
\begin{enumerate}
\item[(a)]
$(\cdot,\cdot)\colon T/T_0\ \times\ K\to \dbZ/q$, $(\sig T_0,\psi)\mapsto\psi(\sig)$;
\item[(b)]
$\langle\cdot,\cdot\rangle\colon T/T_0\times K'\to \dbZ/q$,
$\langle\sig T_0,\varphi\rangle=(\trg\inv(\varphi))(\sig)$.
\end{enumerate}
\end{prop}
\begin{proof}
(a)\quad
A perfect pairing of abelian groups $(\cdot,\cdot)\colon A\times B\to C$ induces for any $A_0\leq A$
a perfect pairing $(A/A_0)\times \{b\in B\ |\ (A_0,b)=0\}\to C$.
Take $A=T/T^q[T,G]$, $B=H^1(T)^G$, $C=\dbZ/q$, and let $A_0$ be the image of $T_0$ in $A$.
By (\ref{pairing1}), the substitution pairing $A\times B\to \dbZ/q$ is perfect.
Furthermore, for $\psi\in H^1(T)^G$ we have $(A_0,\psi)=0$ if and only if $\psi\in K$.

(b) follows from (a).
\end{proof}

For the connection between the second cohomology and central extensions see e.g.,
\cite{NeukirchSchmidtWingberg}*{Th.\ 1.2.4} or \cite{RibesZalesskii}*{\S6.8}.

\begin{prop}
\label{equivelence for dualities}
Let $A$ be a subgroup of $H^2(G/T)$ and let $A_0$ be a set of generators of
$A\cap\trg(H^1(T)^G)$.
The following conditions are equivalent:
\begin{enumerate}
\item[(a)]
$K=\trg\inv[A]$;
\item[(b)]
there is an exact sequence
\[
0\to K\xrightarrow{\trg} A\xrightarrow{\inf} H^2(G);
\]
\item[(c)]
there is an exact sequence
\[
0\to K'\hookrightarrow A\xrightarrow{\inf} H^2(G);
\]
\item[(d)]
$\langle\cdot,\cdot\rangle$ induces a perfect pairing
\[
T/T_0\times \Ker\bigl(A\xrightarrow{\inf} H^2(G)\bigr)\to \dbZ/q;
\]
\item[(e)]
$T_0$ is the annihilator of $A\cap\trg(H^1(T)^G)$ under $\langle\cdot,\cdot\rangle$;
\item[(f)]
$T_0=\bigcap\Ker(\Psi)$, where $\Psi$ ranges over all homomorphisms making a commutative diagram
\[
\xymatrix{
&&&&   G  \ar[dl]_{\Psi} \ar[d] \\
\omega: &0\ar[r] &  \dbZ/q\ar[r] &   C\ar[r]  & G/T\ar[r] & 1,
}
\]
where $\omega$ is the central extension corresponding to some $\varphi\in A_0$
(where an empty intersection is interpreted as $T$).
\end{enumerate}
\end{prop}
\begin{proof}
The 5-term sequence gives (a)$\Leftrightarrow$(b) and (d)$\Leftrightarrow$(e).

The equivalence (b)$\Leftrightarrow$(c) follows from the bijectivity of $\trg\colon K\to K'$.

For (c)$\Leftrightarrow$(d) use Proposition \ref{generalized basic duality}(b).

For (e)$\Leftrightarrow$(f) let $\varphi\in A_0$ and let $\omega$ be the corresponding central extension as above.
By \cite{Hoechsmann68}*{2.1}, for every $\psi\in H^1(T)^G$ with $\trg(\psi)=\varphi$
there is a homomorphism $\Psi\colon G\to C$ such that $\psi=\Psi|_T$ and the diagram in (f) is commutative.
For such $\Psi,\psi$ and for $\sig\in T$ and $\varphi\in A_0$ we have
$\langle\sig T_0,\varphi\rangle=\psi(\sig)=\Psi(\sig)$.
Moreover, $\Ker(\Psi)\leq T$.
We conclude that the annihilator of $A\cap\trg(H^1(T)^G)$ in $T$ under $\langle\cdot,\cdot\rangle$ is $\bigcap\Ker(\Psi)$.
\end{proof}

When these conditions are satisfied, we  say that $A$ is \textbf{dual} to $(T,T_0)$.

\begin{examples}
\label{examples of duality}
\rm

(1)\quad
In \S\ref{section on free groups} we will show that, when $G^{[2]}\isom (\dbZ/q)^I$ for some $I$,
$H^2(G^{[2]})_\dec$ is dual to $(G^{(2)},G_{(3)})$.

\medskip

(2)\quad
For $i\geq2$, $H^2(G^{[i]})$ is dual to $(G^{(i)},G^{(i+1)})$.
Indeed, here $K=H^1(G^{(i)})^G$ \cite{CheboluEfratMinac}*{Lemma 5.4}, so by
the 5-term sequence, (b) holds.

\medskip

(3)\quad
For $\bar G=\dbZ/q$ and for an isomorphism $\psi\in H^1(\bar G)$,
the central extension corresponding to $\beta_{\dbZ/q}(\psi)$ is
\[
0\to\dbZ/q\to\dbZ/q^2\to\dbZ/q\to0
\]
(see \cite{EfratMinac11}*{Prop.\ 9.2}).
For $\bar G=(\dbZ/q)^I$, and for the projection $\psi\colon\bar G\to\dbZ/q$ on the $i_0$th coordinate,
the extension corresponding to $\beta_{\bar G}(\psi)$ is then
\[
0\to \dbZ/q\to\prod_{i\in I}C_i\to \bar G\to 1,
\]
where $C_{i_0}=\dbZ/q^2$ and $C_i=\dbZ/q$ for $i\neq i_0$, with the natural maps \cite{EfratMinac11}*{Lemma 6.2}.

Now assume that $G$ is a profinite group with $G^{[2]}\isom (\dbZ/q)^I$ for some $I$.
Take $T=G^{(2)}$ and $A=A_0=\Img(\beta_{G^{[2]}})$.
Note that then every homomorphism $\Psi$ as in (e) is necessarily surjective, and $\Ker(\Psi)\leq T=G^{(2)}$.
We deduce that the intersection in (f) is
\[
G^{(2)}\cap\bigcap\{M\trianglelefteq G\ |\ G/M\isom\dbZ/q^2\}=G^{q^2}[G,G].
\]
Thus  $A$ is dual to $(G^{(2)},G^{q^2}[G,G])$.
\end{examples}

\begin{prop}
\label{equality of kernels}
Suppose that $A$ is dual to $(T,T_0)$.
Then the kernels of $\inf_{G/T_0}\colon A\to H^2(G/T_0)$, $\inf_G\colon A\to H^2(G)$ coincide.
\end{prop}
\begin{proof}
This is straightforward from condition (c) of Proposition \ref{equivelence for dualities}.
\end{proof}

\section{Cohomological duality triples}
\label{section on cohomological duality triples}
In this section we axiomatize several important cases of functorial subgroups of profinite groups, so that we can
treat them all in a unified way.
For a profinite group $G$ we set $H^{\tensor*}(G)=\bigoplus_{r=0}^\infty H^1(G)^{\tensor r}$, considered as an \textsl{abelian group}.
Assume that we are given:
\begin{enumerate}
\item[(i)]
subfunctors $T, T_0$ of the identity functor on the category of profinite groups;
\item[(ii)]
a natural transformation $\alp$ from the functor $G\mapsto H^{\tensor*}(G)$
to the functor $G\mapsto H^2(G)$ (both from the category of profinite groups to the category of abelian groups).
\end{enumerate}
In other words, for every profinite group $G$ we are given subgroups $T(G),T_0(G)$ of $G$ and a group homomorphism
$\alp_G\colon H^{\tensor*}(G)\to H^2(G)$,
and for every homomorphism $\pi\colon G_1\to G_2$ of profinite groups
there are commutative squares
\[
\xymatrix{
T(G_1) \ar@{^{(}->}[r]\ar[d]_{T(\pi)} & G_1\ar[d]^{\pi} & T_0(G_1) \ar@{^{(}->}[r]\ar[d]_{T_0(\pi)} & G_1\ar[d]^{\pi} &
H^{\tensor*}(G_2) \ar[d]_{\pi^*} \ar[r]^{\alp_{G_2}}  &  H^2(G_2) \ar[d]^{\pi^*_2}
      \\
T(G_2)\ar@{^{(}->}[r] &G_2 & T_0(G_2) \ar@{^{(}->}[r] & G_2    &
H^{\tensor*}(G_1)  \ar[r]^{\alp_{G_1}} & H^2(G_1).
}
\]
We denote
\[
K(G)=\Ker\bigl(\res\colon H^1(T(G))^G\to H^1(T_0(G))\bigr), \qquad
A(G)=\Img(\alp_G).
\]
Observe that, in the previous setup, $\pi^*_2(A(G_2))\subseteq A(G_1)$.

A \textbf{cover} of a profinite group $G$ (relative to $T$) will be an epimorphism $\pi\colon S\to G$,
where $S$ is a profinite group such that $H^2(S)=0$ and the induced map $S/T(S)\to G/T(G)$ is an isomorphism.

\begin{exam}
\label{covers}
\rm
For a profinite group $G$ let $T(G)=G^{(2)}$.
The existence of a cover $S\to G$ means that $G^{[2]}\isom(\dbZ/q)^I$.
Indeed, for such $G$ take $S$ to be a free profinite group of the appropriate rank \cite{NeukirchSchmidtWingberg}*{3.5.4}.
Conversely, a cover $\pi\colon S\to G$ induces an epimorphism $\pi(p)\colon S(p)\to G(p)$ of the maximal pro-$p$ quotients.
Then $H^2(S(p))=0$ \cite{CheboluEfratMinac}*{Lemma 6.5}, so $S(p)$ is a free pro-$p$ group \cite{NeukirchSchmidtWingberg}*{Prop.\ 3.5.17}.
Thus $G^{[2]}\isom S^{[2]}\isom S(p)^{[2]}\isom (\dbZ/q)^I$ for some $I$.
\end{exam}

We call $(T,T_0,\alp)$ a \textbf{cohomological duality triple} if for every profinite group $G$ the following conditions hold:
\begin{enumerate}
\item[(A1)]
$T(G),T_0(G)$ are normal subgroups of $G$;
\item[(A2)]
$T(G)^q[T(G),G]\leq T_0(G)\leq T(G)\leq G^{(2)}$;
\item[(A3)]
for every epimorphism $\pi\colon G\to\bar G$ one has $T(\bar G)=\pi(T(G))$ and $T_0(\bar G)=\pi(T_0(G))$;
\item[(A4)]
if there is a cover $S\to G$, then $A(G/T(G))$ is dual to $(T(G),T_0(G))$.
\end{enumerate}

We list three basic examples of cohomological duality triples.
Condition (A2) for example (2) was shown in Remark \ref{remark on the def of G3}(2), and
(A4) for all these examples is just Examples \ref{examples of duality}.
The verification of the remaining conditions is straightforward.

\begin{examples}
\label{examples of cohomological duality triples}
\rm
(1)\quad
Let $T(G)=G^{(2)}$, $T_0(G)=G_{(3)}$, and let $\alp_G$ be the cup product on $H^1(G)^{\tensor2}$ and
the trivial map on $H^1(G)^{\tensor r}$ for $r\neq2$.
Thus $A(G)=H^2(G)_\dec$.

\medskip

(2)\quad
Let $T(G)=G^{(2)}$, $T_0(G)=G^{(3)}$,  and let
$\alp_G$ be $\beta_G$ on $H^1(G)$, the cup product on $H^1(G)^{\tensor2}$, and
the trivial map on $H^1(G)^{\tensor r}$ for $r\neq1,2$.
Then $A(G)=\Img(\beta_G)+H^2(G)_\dec$.
If there is a cover $S\to G$, then by Lemma \ref{H2 for elementary abelian groups} and Example \ref{covers},
$A(G^{[2]})=H^2(G^{[2]})$,
so indeed, Example \ref{examples of duality}(2) applies.

\medskip

(3)\quad
Let $T(G)=G^{(2)}$, $T_0(G)=G^{q^2}[G,G]$,
and  let $\alp_G=\beta_G$ on $H^1(G)$ and $\alp_G=0$ on $H^1(G)^{\tensor r}$ for $r\neq1$.
Thus $A(G)=\Img(\beta_G)$.
\end{examples}

\begin{rem}
\rm
Let $(T,T_0,\alp)$ be a cohomological duality triple, $\pi\colon G_1\to G_2$ an epimorphism of profinite groups,
and suppose that there are covers $S\to G_i$, $i=1,2$, which commute with $\pi$.
Then $\pi$ induces an isomorphism $G_1/T(G_1)\xrightarrow{\sim}G_2/T(G_2)$.
As $T(G_i)\leq G_i^{(2)}$, the induced map $H^{\tensor*}(G_2)\to H^{\tensor*}(G_1)$ is therefore an isomorphism.
Hence $\pi^*_2(A(G_2))=A(G_1)$.
\end{rem}

For a cohomological duality triple $(T,T_0,\alp)$ and for $r\geq0$, $t\geq1$,
let $C_{r,t}(G)$ be the $\dbZ/q$-submodule of $H^1(G)^{\tensor r}$ generated by all its
elements $\psi_1\tensor\cdots\tensor\psi_r$ such that $\alp_G(\psi_{i_1}\tensor\cdots\tensor\psi_{i_t})=0$
for some $1\leq i_1<\cdots<i_t\leq r$.
Thus $C_{r,t}(G)=0$ for $r<t$.
We define $H^r_{t,\alp}(G)=H^1(G)^{\tensor r}/C_{r,t}(G)$ and  a graded ring
$H^*_{t,\alp}(G)=\bigoplus_{r=0}^\infty H^r_{t,\alp}(G)$.
In particular, $H^0_{t,\alp}(G)=\dbZ/q$.
Since $\alp$ is a natural transformation, $H^*_{t,\alp}$ is a functor.
Note that $\alp_G$ induces a homomorphism $\bar\alp_G^t\colon H^t_{t,\alp}(G)\to A(G)$.

\begin{examples}
\label{examples of alpha bar}
\rm
\begin{enumerate}
\item[(1)]
In Example \ref{examples of cohomological duality triples}(1),
$H^*_{2,\alp}(G)=\widehat{H^*(G)}$.
The map $\bar\alp_G^2\colon H^2_{2,\alp}(G)\to A(G)=H^2(G)_\dec$ is surjective, and is injective if and only if $H^*(G)$
is $2$-quadratic.
\item[(2)]
In Example \ref{examples of cohomological duality triples}(2),
$H^*_{1,\alp}(G)=H_\Bock^*(G)$ and $H^*_{2,\alp}(G)=\widehat{H^*(G)}$.
\item[(3)]
In Example \ref{examples of cohomological duality triples}(3),
$H^*_{1,\alp}(G)=H^*_\Bock(G)$ and the map $\bar\alp_G^1\colon H^1_{1,\alp}(G)=H^1(G)/\Ker(\beta_G)\to A(G)=\Img(\beta_G)$
is an isomorphism.
\end{enumerate}
\end{examples}

\section{Quotients that determine cohomology}
\label{section on quotients that determine cohomology}
This section is devoted to proving Theorem A.
More generally, let $(T,T_0,\alp)$ be a cohomological duality triple.
We show that, assuming the existence of a cover $S\to G$, the quotient $G/T_0(G)$ determines the graded rings
$H^*_{t,\alp}(G)$, $t\geq1$,
and is in fact the minimal such quotient:

\begin{thm}
\label{first main thm}
Assume that there is a cover $S\to G$.
Let $N$ be a normal subgroup of $G$ contained in $T(G)$, and consider the following conditions:
\begin{enumerate}
\item[(a)]
$N\leq T_0(G)$;
\item[(b)]
$\inf_G\colon A(G/N)\to A(G)$ is an isomorphism;
\item[(c)]
$\inf_G\colon H^*_{t,\alp}(G/N)\to H^*_{t,\alp}(G)$ is an isomorphism for every $t$.
\end{enumerate}
Then (a)$\Leftrightarrow$(b)$\Rightarrow$(c).
Moreover, if there exist $r\geq1$ such that
$\bar\alp_G^r\colon H^r_{r,\alp}(G)\to A(G)$ is injective and
$\bar\alp_{G/N}^r\colon H^r_{r,\alp}(G/N)\to A(G/N)$ surjective, then (a)--(c) are equivalent.
\end{thm}

Before we proceed with the proof of Theorem \ref{first main thm},
we apply it to Examples \ref{examples of cohomological duality triples},
to deduce the following special cases.

\begin{cor}
\label{ex quoteints determine cohomology}
Let $G$ be a profinite group with $G^{[2]}\isom(\dbZ/q)^I$ for some $I$.
Let $N$ be a normal subgroup of $G$ contained in $G^{(2)}$.
\begin{enumerate}
\item[(1)]
If $H^*(G)$ is $2$-quadratic, then the following conditions are equivalent:
\begin{enumerate}
\item[(a)]
$N\leq G_{(3)}$;
\item[(b)]
$\inf_G\colon H^2(G/N)_\dec\to H^2(G)_\dec$ is an isomorphism;
\item[(c)]
$\inf_G\colon H^*(G/N)_\dec\to H^*(G)_\dec$ is an isomorphism.
\end{enumerate}
\item[(2)]
The following conditions are equivalent:
\begin{enumerate}
\item[(a)]
$N\leq G^{(3)}$;
\item[(b)]
$\inf_G\colon \Img(\beta_{G/N})+H^2(G/N)_\dec\longrightarrow \Img(\beta_G)+H^2(G)_\dec$ is an isomorphism;
\end{enumerate}
\item[(3)]
The following conditions are equivalent:
\begin{enumerate}
\item[(a)]
$N\leq G^{q^2}[G,G]$;
\item[(b)]
$\inf_G\colon \Img(\beta_{G/N})\to \Img(\beta_G)$ is an isomorphism;
\item[(c)]
$\inf_G\colon H^*(G/N)_\Bock\to H^*(G)_\Bock$ is an isomorphism.
\end{enumerate}
\end{enumerate}
\end{cor}
\begin{proof}
Everything follows directly from Theorem \ref{first main thm},
using Example \ref{covers} and Examples \ref{examples of alpha bar}, except for
the equivalence with (1)(c).
Since 1(c) trivially implies 1(b), it suffices to show that 1(b) implies 1(c).
Indeed, by Theorem \ref{first main thm} and 1(b), $\inf_G\colon H^*_{t,\alp}(G/N)\to H^*_{t,\alp}(G)$
is an isomorphism for every $t$.
For $t=2$ this means that $\inf_G\colon \widehat{H^*(G/N)}\to \widehat{H^*(G)}$ is bijective.
Using the functoriality of $\omega$ we see that the square below is commutative:
\[
\xymatrix{
\widehat{H^*(G/N)}\ar[r]^{\inf_G}_{\sim}\ar@{->>}[d]_{\omega_{G/N}} & \widehat{H^*(G)}\ar[d]^{\omega_G}_{\wr} \\
H^*(G/N)_\dec \ar[r]^{\inf_G} & H^*(G)_\dec.
}
\]
Since $\omega_G$ is by assumption bijective, the lower inflation is bijective.
\end{proof}

\begin{proof}[Proof of Theorem A]
In view of Theorem \ref{absolute Galois groups}, part (1) of Corollary \ref{ex quoteints determine cohomology} implies Theorem A.
\end{proof}

Note that in (1) the equivalence (a)$\Leftrightarrow$(b) holds even without the $2$-quadraticness assumption.

For the proof of Theorem \ref{first main thm} we first show:

\begin{prop}
\label{ppp}
Let $\pi_i\colon S\to G_i$, $i=1,2$, be covers  and $\pi\colon G_1\to G_2$ an epimorphism
with $\pi\circ\pi_1=\pi_2$.
The following conditions are equivalent:
\begin{enumerate}
\item[(a)]
the induced map $G_1/T_0(G_1)\to G_2/T_0(G_2)$ is an isomorphism;
\item[(b)]
$\Ker(\pi)\leq T_0(G_1)$;
\item[(c)]
the induced map $A(\pi)\colon A(G_2)\to A(G_1)$ is an isomorphism;
\item[(d)]
the induced map $A(\pi)\colon A(G_2)\to A(G_1)$ is a monomorphism.
\end{enumerate}
\end{prop}
\begin{proof}
(a)$\Rightarrow$(b), (c)$\Rightarrow$(d):\quad
Trivial.

\medskip

(b)$\Rightarrow$(a):\quad
By (A3), $T_0(G_2)=\pi(T_0(G_1))$.
Hence the kernel of the induced epimorphism
$G_1/T_0(G_1)\to G_2/T_0(G_2)$ is $\Ker(\pi)T_0(G_1)/T_0(G_1)$, which is trivial by (b).

\medskip

(d)$\Rightarrow$(a):\quad
The map $\bar\pi\colon G_1/T(G_1)\to G_2/T(G_2)$ induced by $\pi$ is an isomorphism.
Hence so is $\pi^*_1\colon H^1(G_2)\to H^1(G_1)$.
By (A4), $\pi$ induces a commutative diagram with exact rows
\[
\xymatrix{
0\ar[r] & K(G_2)\ar[r]^{\trg\qquad}\ar[d] & A(G_2/T(G_2))\ar[r]^{\inf}\ar[d]^{\wr}_{A(\bar\pi)} &
          \ A(G_2)\ \ar@{>->}[d]^{A(\pi)} \\
0\ar[r] & K(G_1)\ar[r]^{\trg\qquad} & A(G_1/T(G_1))\ar[r]^{\inf} & A(G_1).
}
\]
By the assumptions and the snake lemma, the left vertical map is an isomorphism.
Passing to duals using Proposition \ref{generalized basic duality}(a), we obtain that the induced map $T(G_1)/T_0(G_1)\to T(G_2)/T_0(G_2)$ is an isomorphism.
Now apply the snake lemma for the commutative diagram with exact rows
\[
\xymatrix{
1\ar[r]& T(G_1)/T_0(G_1)\ar[r]\ar[d]_{\wr} & G_1/T_0(G_1)\ar[r]\ar[d] & G_1/T(G_1)\ar [r]\ar[d]^{\wr} & 1\qquad \\
1\ar[r]& T(G_2)/T_0(G_2)\ar[r] & G_2/T_0(G_2)\ar[r] & G_2/T(G_2)\ar [r] & 1.
}
\]

\medskip

(a)$\Rightarrow$(c):\quad
Consider the commutative diagram
\[
\xymatrix{
A(G_2/T(G_2))\ar[r]^{\inf}\ar[d] & A(G_2/T_0(G_2))\ar[r]^{\qquad\inf}\ar[d] & A(G_2)\ar[d]^{A(\pi)} \\
A(G_1/T(G_1))\ar[r]^{\inf} & A(G_1/T_0(G_1))\ar[r]^{\qquad\inf} & A(G_1),
}
\]
where the vertical maps are induced by $\pi$.
(a) implies that the middle vertical map is an isomorphism.
Since $T_0(G_i)\leq T(G_i)\leq G_i^{(2)}$, the inflations $H^1(G_i/T(G_i))\to H^1(G_i/T_0(G_i))\to H^1(G_i)$
are isomorphisms, $i=1,2$.
Hence the horizontal inflation maps are surjective.
By Proposition \ref{equality of kernels}, in each row the kernel of the left inflation map equals
the kernel of the composed inflation map.
It follows that $A(\pi)$ is an isomorphism.
\end{proof}

\begin{prop}
\label{condition e}
In the setup of Proposition \ref{ppp}, conditions (a)--(d) imply:
\begin{enumerate}
\item[(e)]
$\pi$ induces for every $t\geq1$ an isomorphism $\pi^*\colon H^*_{t,\alp}(G_2)\to H^*_{t,\alp}(G_1)$.
\end{enumerate}
Moreover, if there exists $r\geq1$ with
$\bar\alp_{G_1}^r\colon H^r_{r,\alp}(G_1)\to A(G_1)$
injective and $\bar\alp_{G_2}^r\colon H^r_{r,\alp}(G_2)\to A(G_2)$ surjective,
then (e) for $t=r$ is equivalent to (a)--(d).
\end{prop}
\begin{proof}
Since the induced map $G_1^{[2]}\to G_2^{[2]}$ is an isomorphism, so is $\pi_1^*\colon H^1(G_2)\to H^1(G_2)$,
and we obtain a commutative square
\[
\xymatrix{
H^1(G_2)^{\tensor t} \ar[r]^{(\pi^*_1)^{\tensor t}}_{\sim}\ar[d]_{\alp_{G_2}} & H^1(G_1)^{\tensor t}\ar[d]^{\alp_{G_1}}\\
A(G_2) \ar[r]^{A(\pi)} & A(G_1).
}
\]
Assuming (d),  $(\pi^*_1)^{\tensor t}$ maps $C_{r,t}(G_2)$ bijectively onto $C_{r,t}(G_1)$ for every $t\geq1$,
and therefore $\pi^*\colon H^*_{t,\alp}(G_2)\to H^*_{t,\alp}(G_1)$ is an isomorphism.

For the last assertion, assume (e) and consider the commutative square
\[
\xymatrix{
H^r_{r,\alp}(G_2) \ar@{->>}[d]_{\bar\alp_{G_2}^r}\ar[r]^{\pi^*_r}_{\sim} &
H^r_{r,\alp}(G_1)\ar@{>->}[d]^{\bar\alp_{G_1}^r} \\
A(G_2)\ar[r]^{A(\pi)} & A(G_1).
}
\]
The assumptions imply that $A(\pi)$ is injective.
\end{proof}

\begin{proof}[Proof of Theorem \ref{first main thm}]
As $N\leq T(G)$, (A3) gives $T(G/N)=T(G)/N$.
Therefore the composed map $S\to G\to G/N$ is a cover.
Now apply Propositions \ref{ppp} and \ref{condition e} for the projection $\pi\colon G\to G/N$.
\end{proof}

\section{Isomorphisms}
\label{section on isomorphisms}
We now apply the results of \S\ref{section on quotients that determine cohomology} to the case of pro-$p$ groups.

\begin{prop}
\label{pro-p}
Let $(T,T_0,\alp)$ be a cohomological duality triple.
Let $\pi_i\colon S\to G_i$, $i=1,2$, be covers  and $\pi\colon G_1\to G_2$ an epimorphism of pro-$p$ groups
with $\pi\circ\pi_1=\pi_2$.
Suppose that $A(G_2)=H^2(G_2)$.
Then $\pi$ is an isomorphism if and only if the induced map $G_1/T_0(G_1)\to G_2/T_0(G_2)$ is an isomorphism.
\end{prop}
\begin{proof}
The ``only if" part follows from (A3).
For the ``if" part, recall that by \cite{Serre65}*{Lemma 2}, $\pi$ is an isomorphism if and only if
$\pi^*_r\colon H^r(G_2)\to H^r(G_1)$ is an isomorphism for $r=1$ and a monomorphism for $r=2$.
As $\pi_1^*$ commutes with the isomorphisms $H^1(G_i)\to H^1(S)$, $i=1,2$, it is also an isomorphism.
Since the induced map $G_1/T_0(G_1)\to G_2/T_0(G_2)$ is an isomorphism and by Proposition \ref{ppp},
$\pi^*_2\colon A(G_2)=H^2(G_2)\to H^2(G_1)$ is a monomorphism.
Hence $\pi$ is an isomorphism.
\end{proof}

We deduce the following strengthening of \cite{CheboluEfratMinac}*{Remark 6.4, Th.\ D}
(which deal with the quotients $G_i^{[3]}$):

\begin{cor}
\label{first equivalnce in thmC}
Let $\pi\colon G_1\to G_2$ be an epimorphism of pro-$p$ groups inducing an isomorphism
$\pi^{[2]}\colon G_1^{[2]}\xrightarrow{\sim}G_2^{[2]}\isom(\dbZ/q)^I$
for some $I$, and such that  $H^2(G_2)=H^2(G_2)_\dec$.
Then $\pi$ is an isomorphism if and only if the induced map $\pi_{[3]}\colon (G_1)_{[3]}\to (G_2)_{[3]}$ is an isomorphism.
\end{cor}
\begin{proof}
Choose bases (i.e., generating subsets converging to $1$) $\bar Z_i$ of $G_i^{[2]}$, $i=1,2$,
such that $\pi^{[2]}(\bar Z_1)=\bar Z_2$.
Lift $\bar Z_1$ to a subset $Z_1$ of $G_1$, and let $Z_2=\pi(Z_1)$.
By the Frattini argument, $Z_1,Z_2$ generate $G_1,G_2$, respectively.
Let $S$ be a free pro-$p$ group with basis $Z_1$.
Let $\pi_1\colon S\to G_1$ be the natural epimorphism, let $\pi_2=\pi\circ\pi_1$, and note that $\pi_1,\pi_2$ are covers.
Now take the triple of Example \ref{examples of cohomological duality triples}(1) and apply Proposition \ref{pro-p}.
\end{proof}

\begin{proof}[Proof of Theorem C]
In the first equivalence, the ``only if" part is immediate.
For the ``if" part note that, by Remark \ref{GF(p)}, $G_{F_1}(p)^{[2]}\isom G_{F_2}(p)^{[2]}\isom(\dbZ/q)^I$ for some set $I$.
Hence we may apply Corollary 6.2.

For the second equivalence, consider the cohomological duality triple of Example \ref{examples of cohomological duality triples}(1), and let $G_i=G_{F_i}(p)$, $i=1,2$.
Then $\bar\alp^2_{G_1}\colon H^2_{2,\alp}(G_1)\to H^2_\dec(G_1)$ is injective and $\bar\alp^2_{G_2}\colon H^2_{2,\alp}(G_2)\to H^2_\dec(G_2)$ is surjective (see Example \ref{examples of alpha bar}(1)).
Hence we may apply Proposition \ref{condition e} with $r=2$ to conclude that the induced map $\pi_{[3]}\colon (G_1)_{[3]}\to (G_2)_{[3]}$ is an isomorphism if and only if the map $\pi^*\colon H^*_{2,\alp}(G_2)\to H^*_{2,\alp}(G_1)$ is an isomorphism.
It remains to recall that $H^*_{2,\alp}(G_i)=\widehat{H^*(G_i)}=H^*(G_i)$, by Example \ref{examples of alpha bar}(1) and the fact that $H^*(G_i)$ is quadratic.
\end{proof}

Next Corollary \ref{ex quoteints determine cohomology}(1) and Remark \ref{GF(p)}
give the following refinement of \cite{CheboluEfratMinac}*{Prop.\ 9.1}.

\begin{cor}
\label{only one Galois group}
Let $G_1,G_2$ be profinite groups such that $(G_1)_{[3]}\isom (G_2)_{[3]}$ but $H^*(G_1)\not\isom H^*(G_2)$.
Then at most one of $G_1,G_2$ can be isomorphic to the maximal pro-$p$ Galois group $G_F(p)$ of a field $F$
containing a root of unity of order $q$.
\end{cor}

As in \cite{CheboluEfratMinac}*{\S9}, Corollary \ref{only one Galois group} can be used to show
that various pro-$p$ groups do not occur as $G_F(p)$ for $F$ as above.

\begin{examples}
\label{examples}
\rm
We assume that $q=p$ is prime.

\medskip

(1)\quad
Let $S$ be a free pro-$p$ group and $R$  a normal subgroup of $S$
such that $R\leq S_{(3)}$ and $S\not\isom S/R$.
Take in Corollary \ref{only one Galois group} $G_1=S$ and $G_2=S/R$.
Then $(G_1)_{[3]}\isom (G_2)_{[3]}$ and $H^1(G_1)\isom H^1(G_2)$  (as $R\leq S^{(2)}$).
Hence $G_1,G_2$ have the same rank \cite{NeukirchSchmidtWingberg}*{Prop.\ 3.9.1}.
Since a free pro-$p$ group is determined by its rank, $G_2$  is not free pro-$p$.
Therefore $H^2(S)=0\neq H^2(G_2)$  \cite{NeukirchSchmidtWingberg}*{Cor. 3.9.5}.
Now $G_1$ is realizable as an absolute Galois group of a field
of characteristic $\neq p$ \cite{FriedJarden}*{Cor.\ 23.1.2}, and such a field automatically
contains a $p$th root of unity.
By Corollary \ref{only one Galois group}
$G_2\not\isom G_F(p)$ for any field $F$ containing a $p$th root of unity.

\medskip
(2)\quad
Take in (1)  $G_1=S=\dbZ_p$ and $R=(\dbZ_p)_{(3)}=\del p\dbZ_p$
(with $\del$ as in \S\ref{section on cohomology}).
Thus $G_2=S/R=\dbZ/4$ for $p=2$, and $G_2=\dbZ/p$ for $p$ odd.
Consequently, $G_2\not\isom G_F(p)$ for any field $F$ containing a $p$th root of unity.
We recover Becker's generalization of the classical Artin--Schreier theorem \cite{Becker74}:
the order of an element in $G_F(p)$ can be only $1$, $2$, or $\infty$.
Observe that if one used \cite{CheboluEfratMinac}*{\S9} instead of Corollary  \ref{only one Galois group}, one would get only that, for $p$ odd, $\dbZ/p^2$ is not a maximal pro-$p$ Galois group of a field as above.
\end{examples}

\section{Quotients determined by cohomology}
\label{section quotients determined by cohomology}
In this section we prove a partial converse of Theorem \ref{first main thm},
saying that for a cohomological duality triple $(T,T_0,\alp)$, and assuming the existence of covers,
$G/T_0(G)$ is determined by $\alp_G$ and $G/T(G)$ (Theorem \ref{second main thm}).
In particular, this will prove Theorem B.

First consider a cover $\pi\colon S\to G$.
Let $R=\Ker(\pi)$.
Then $R\leq T(S)$.
In view of Proposition \ref{generalized basic duality}(a), there is a commutative diagram of perfect (substitution) pairings
\begin{equation}
\label{rrrr}
\xymatrix{
R/R^q[R,S] \ar[d]_{\iota} & *-<3pc>{\times} & H^1(R)^S \ar[r]  & \dbZ/q \ar@{=}[d]\\
T(S)/T_0(S) & *-<3pc>{\times} &  K(S) \ar[r]\ar[u]_{\res_R} & \dbZ/q.
}
\end{equation}
Also, there is a commutative diagram with exact rows
\begin{equation}
\label{utysz}
\xymatrix{
1\ar[r] &  R\ar[r]\ar@{^{(}->}[d] &  S \ar[r]^{\pi}\ar@{=}[d] & G\ar[r]\ar@{->>}[d] & 1 \\
1 \ar[r] & T(S) \ar[r] & S \ar[r] & S/T(S) \ar[r] & 1.
}
\end{equation}
Since $R\leq T(S)\leq S^{(2)}$, the inflation maps $H^1(G)\to H^1(S)$, $H^1(S/T(S))\to H^1(S)$ are isomorphisms.
By the $5$-term sequence and as $H^2(S)=0$, the two transgression maps arising from (\ref{utysz}) are therefore isomorphisms.
By the functoriality of transgression \cite{EfratMinac11}*{(2.2)}, they commute.
We get a commutative diagram
\begin{equation}
\label{gzzsk}
\xymatrix{
K(S) \ar@{^{(}->}[r]\ar[d]_{\trg}^{\wr}  & H^1(T(S))^S \ar[r]^{\res}\ar[d]_{\trg}^{\wr} & H^1(R)^S \ar[d]^{\trg}_{\wr} \\
A(S/T(S))\ar@{^{(}->}[r] & H^2(S/T(S)) \ar[r]^{\quad\inf} & H^2(G)\\
}
\end{equation}
where the left isomorphism is by (A4).
Let $g$ be the composite map
$K(S)\to H^2(G)$ arising from this diagram.
Let $\Ker(g)^\vee$ denote the annihilator of $\Ker(g)$ in $T(S)/T_0(S)$ under the lower pairing of (\ref{rrrr}).

\begin{lem}
\label{G3 and dual}
$G/T_0(G)\isom (S/T_0(S))/\Ker(g)^\vee$.
\end{lem}
\begin{proof}
By (\ref{gzzsk}), $\Ker(g)$ is the kernel of $\res_R\colon K(S)\to H^1(R)^S$.
By (\ref{rrrr}), $\Ker(g)^\vee\isom\Img(\iota)=RT_0(S)/T_0(S)$.
Using (A3) we see that this is the kernel of the induced epimorphism $S/T_0(S)\to G/T_0(G)$.
\end{proof}

Given covers $S\to G_i$, $i=1,2$, we have isomorphisms $S^{[2]}\to G_i^{[2]}$.
They induce isomorphisms $H^1(G_i)\isom H^1(S)$ and $H^{\tensor*}(G_1)\xrightarrow{\sim} H^{\tensor*}(G_2)$.

\begin{thm}
\label{second main thm}
Let $(T,T_0,\alp)$ be a cohomological duality triple.
Let $\pi_i\colon S\to G_i$ be covers, $i=1,2$,
and $\sig\colon H^{\tensor*}(G_1)\to H^{\tensor*}(G_2)$
the induced isomorphism.
Suppose that there is a monomorphism
$\tau\colon H^2(G_1)\to H^2(G_2)$ with $\alp_{G_2}\circ\sig=\tau\circ\alp_{G_1}$.
Then there is an isomorphism $G_1/T_0(G_1)\isom G_2/T_0(G_2)$ compatible with $\pi_i$, $i=1,2$.
\end{thm}
\begin{proof}
For $i=1,2$ there is a commutative diagram
\[
\xymatrix{
H^{\tensor*}(S/T(S)) \ar[r]_{\sim\ }\ar@{->>}[d]_{\alp_{S/T(S)}} &
 H^{\tensor*}(G_i/T(G_i)) \ar[r]^{\quad\inf}_{\quad\sim}
 \ar@{->>}[d]_{\alp_{G_i/T(G_i)}}  &
  H^{\tensor*}(G_i)  \ar@{->>}[d]^{\alp_{G_i}} \\
A(S/T(S)) \ar[r]_{\sim} & A(G_i/T(G_i)) \ar[r]^{\quad\inf} & A(G_i).
}
\]
Define a homomorphism $g_i\colon K(S)\to H^2(G_i)$ as above.

Given $\gam\in H^{\tensor*}(S/T(S))$ let $\hat\gam_i$ be the corresponding element in $H^{\tensor*}(G_i)$.
Then $\gam$ maps trivially to $A(G_i)$ if and only if $\alp_{G_i}(\hat\gam_i)=0$.
Our assumption implies that $\alp_{G_1}(\hat\gam_1)=0$ if and only if $\alp_{G_2}(\hat\gam_2)=0$.
Consequently, the kernels of $\inf\colon A(S/T(S))\to A(G_i)\subseteq H^2(G_i)$, $i=1,2$, coincide.
Their preimages  in $K(S)$ under transgression are $\Ker(g_i)$, $i=1,2$ (see (\ref{gzzsk})),
which therefore also coincide.
Now use Lemma \ref{G3 and dual}.
\end{proof}

Applying this to Examples \ref{examples of cohomological duality triples} we obtain:

\begin{cor}
\label{ex cohomology determines quotients}
\rm
Assume that $G^{[2]}\isom(\dbZ/q)^I$ for some $I$.

\medskip

(1)\quad
$G^{[2]}$ and $\cup\colon H^1(G)\times H^1(G)\to H^2(G)$ determine $G_{[3]}=G/G_{(3)}$.

\medskip

(2)\quad
$G^{[2]}$, $\beta_G$, and $\cup\colon H^1(G)\times H^1(G)\to H^2(G)$ determine $G^{[3]}=G/G^{(3)}$.

\medskip

(3)\quad
$G^{[2]}$ and $\beta_G$ determine $G/G^{q^2}[G,G]$.
\end{cor}

\begin{proof}[Proof of Theorem B]
Use Theorem \ref{absolute Galois groups} and part (1) of Corollary \ref{ex cohomology determines quotients}.
\end{proof}

\section{Presentations}
Let $(T,T_0,\alp)$ be again a cohomological duality triple.
We use the techniques of the previous section to characterize the surjectivity of $\alp_G$ in terms of group presentations.
Let again $\pi\colon S\to G$ be a cover and $R=\Ker(\pi)$.
Note that $R^q[R,S]\leq R\cap T_0(S)$, by (A2).

\begin{thm}
\label{5th main thm}
There is a natural duality between $(R\cap T_0(S))/R^q[R,S]$ and the cokernel of $\inf_G\colon A(G/T(G))\to H^2(G)$.
\end{thm}
\begin{proof}
The induced map $S/T(S)\to G/T(G)$ is an isomorphism.
From (\ref{gzzsk}) we obtain a commutative diagram
\[
\xymatrix{K(S)\ar[r]^{\res_R}\ar[d]_{\trg}^{\wr} & H^1(R)^S\ar[d]_{\wr}^{\trg} \\
A(S/T(S))\ar[r]^{\inf} & H^2(G).
}
\]
The right transgression maps $\Coker(\res_R)$ isomorphically
onto the cokernel of $\inf_G\colon A(G/T(G))\to H^2(G)$.
By (\ref{rrrr}), $\Coker(\res_R)$ is dual to $\Ker(\iota)=(R\cap T_0(S))/R^q[R,S]$.
\end{proof}

\begin{cor}
$\alp_G$ is onto $H^2(G)$ if and only if $R^q[R,S]=R\cap T_0(S)$.
\end{cor}
\begin{proof}
The surjectivity of $\alp_G$ is equivalent to the surjectivity of the inflation $\inf_G\colon A(G/T(G))\to H^2(G)$.
Now use Theorem \ref{5th main thm}.
\end{proof}

Applying this for Examples \ref{examples of cohomological duality triples} we deduce:

\begin{examples}
\rm
Assume that $G^{[2]}\isom(\dbZ/q)^I$ for some $I$.

\medskip

(1)\quad
$H^2(G)=H^2(G)_\dec$ if and only if $R^q[R,S]=R\cap S_{(3)}$.

\medskip

(2)\quad
$H^2(G)=H^2(G)_\dec+\Img(\beta_G)$ if and only if $R^q[R,S]=R\cap S^{(3)}$
(compare also \cite{CheboluEfratMinac}*{Th.\ 7.1}).

\medskip
(3)\quad
$H^2(G)=\Img(\beta_G)$ if and only if $R^q[R,S]=R\cap S^{q^2}[S,S]$.
\end{examples}

By Theorem \ref{absolute Galois groups}, if $G=G_F$ is the absolute Galois group of a field $F$ containing a root of unity of order $q$, then (1), and therefore (2), are valid.

\section{$T_0(G)$ as an intersection}
Let $(T,T_0,\alp)$ be again a cohomological duality triple.
In this section we present $T_0(G)$ as the intersection
of all open normal subgroups $M$ of the profinite group $G$ with $G/M$ contained in a certain
list $\calL(G)$ of finite groups.
In all our main examples, and assuming, say, that $G$ is an absolute Galois group of a field containing a root of unity of order $p$, the list $\calL(G)$ is finite and explicit, and does not depend on $G$.
In particular, this will imply Theorem D.

Following \cite{EfratMinac11}, we say that $G$ has \textbf{Galois relation type} if
\begin{enumerate}
\item[(i)]
$G^{[2]}\isom (\dbZ/q)^I$ for some set $I$;
\item[(ii)]
there exists $\xi\in H^1(G)$ with $\beta_G(\psi)=\psi\cup\xi$ for every $\psi\in H^1(G)$;
\item[(iii)]
the kernel of $\inf\colon H^2(G^{[2]})_\dec\to H^2(G)$ is generated by cup products $\psi\cup\psi'$,
with $\psi,\psi'\in H^1(G^{[2]})$.
\end{enumerate}
By Theorem \ref{absolute Galois groups} and Lemma \ref{ker of inflation}, this holds when $G=G_F$
is the absolute Galois group of a field $F$ containing a root of unity of order $q$
(this was earlier shown in \cite{EfratMinac11}*{Prop.\ 3.2}).

\begin{defin}
\rm
A \textbf{special set} for the profinite group $G$ with respect to $(T,T_0,\alp)$ will be a set $\Sig$
of pairs $(\bar G,\bar\varphi)$ such
that $\bar G$ is a finite quotient of $G^{[2]}$, $\bar\varphi\in H^{\tensor*}(\bar G)$, and the kernel of
$\inf_G\colon A(G/T(G))\to H^2(G)$ is generated by the elements
$\alp_{G/T(G)}(\inf_{G/T(G)}(\bar\varphi))$ with $(\bar G,\bar\varphi)\in\Sig$.
\end{defin}

\begin{examples}
\label{examples of special sets}
\rm
Let $G$ be a profinite group of Galois relation type.

\medskip

(1)\quad
Consider the cohomological duality triple of Example \ref{examples of cohomological duality triples}(1).
Take $\psi,\psi'\in H^1(G^{[2]})$ such that $\psi\cup\psi'=\alp_{G^{[2]}}(\psi\tensor\psi')\neq0$
is in the kernel of  $\inf\colon H^2(G^{[2]})_\dec\to H^2(G)$.
Let $\bar G=G^{[2]}/(\Ker(\psi)\cap\Ker(\psi'))$
and take $\bar\psi,\bar\psi'\in H^1(\bar G)$ with $\psi=\inf_{G^{[2]}}(\bar\psi)$, $\psi'=\inf_{G^{[2]}}(\bar\psi')$.
By (iii), the set of all such pairs $(\bar G,\bar\psi\tensor\bar\psi')$ is a special set for $G$.

\medskip

(2)\quad
Consider the triple of Example \ref{examples of cohomological duality triples}(2).
By (i) and Lemma \ref{H2 for elementary abelian groups},
\[
H^2(G^{[2]})=\Img(\beta_{G^{[2]}})+H^2(G^{[2]})_\dec=A(G/T(G)).
\]

We first claim that $K'=\Ker(\inf_G\colon H^2(G^{[2]})\to H^2(G))$ is generated by elements of the form
$\alp_{G^{[2]}}(-\psi\oplus (\psi\tensor\psi'))=-\beta_{G^{[2]}}(\psi)+\psi\cup\psi'$,
where $\psi,\psi'\in H^1(G^{[2]})$, $\psi\neq0$, and $-\psi\oplus (\psi\tensor\psi')$ is considered as an element of $H^{\tensor*}(G^{[2]})$ (compare \cite{EfratMinac11}*{Prop.\ 4.3}).
Indeed, for $\xi$ as in (ii) we take $\xi_0\in H^1(G^{[2]})$ with $\xi=\inf_G(\xi_0)$.\
Let $\theta=\beta_{G^{[2]}}(\eta)+\sum_{i=1}^n\psi_i\cup\psi'_i\in K'$,
where $\eta,\psi_i,\psi'_i\in H^1(G^{[2]})$.
Then also $\eta\cup\xi_0+\sum_{i=1}^n\psi_i\cup\psi'_i$ is in $K'$, and by (iii),
it can be written as $\sum_{j=1}^m\chi_j\cup\chi'_j$, with $\chi_j,\chi'_j\in H^1(G^{[2]})$ and
$\chi_j\cup\chi'_j\in K'$ for each $j$.
Hence
\[
\begin{split}
\theta&=\beta_{G^{[2]}}(\eta)+(-\eta)\cup\xi_0+\sum_{j=1}^m\chi_j\cup\chi'_j\\
&=(\beta_{G^{[2]}}(\eta)+(-\eta)\cup\xi_0)+\sum_{j=1}^m(-\beta_{G^{[2]}}(\chi_j)+\chi_j\cup(\chi'_j+\xi_0))\\
&\qquad\qquad\qquad\qquad\qquad\quad+\sum_{j=1}^m(\beta_{G^{[2]}}(\chi_j)+(-\chi_j)\cup\xi_0)
\end{split}
\]
and this sum is in $K'$, proving the claim.

Now given $-\psi\oplus (\psi\tensor\psi')$ as above, let $\bar G=G^{[2]}/(\Ker(\psi)\cap\Ker(\psi'))$
and take $\bar\psi,\bar\psi'\in H^1(\bar G)$ with $\psi=\inf_{G^{[2]}}(\bar\psi)$, $\psi'=\inf_{G^{[2]}}(\bar\psi')$.
The set $\Sig$ of all pairs $(\bar G,-\bar\psi+(\bar\psi\tensor\bar\psi'))$ is special for $G$.

\medskip

(3)\quad
Consider the cohomological duality triple of Example \ref{examples of cohomological duality triples}(3).
Trivially, the kernel of $\inf_G\colon \Img(\beta_{G^{[2]}})\to \Img(\beta_G)$
is generated by elements $\beta_{G^{[2]}}(\psi)$, with $\psi\in H^1(G^{[2]})$.
For such $\psi$ let $\bar G=G^{[2]}/\Ker(\psi)$ and take $\bar\psi\in H^1(\bar G)$ with $\psi=\inf(\bar\psi)$.
The set $\Sig$ of all pairs $(\bar G,\bar\psi)$ is special for $G$.
\end{examples}

For the rest of this section we assume that $q=p$ is prime.
When $p\neq2$ let $H_{p^3}$ be the Heisenberg group of order $p^3$ (see the Introduction), and let
\[
M_{p^3}=\bigl\langle r,s\ \bigm|\ r^{p^2}=s^p=1,\ r^p=[r,s]\bigr\rangle
\]
be the unique nonabelian group of odd order $p^3$ and exponent $p^2$.
Let $D_4$ be the dihedral group of order $8$.

Given a special set $\Sig$ for $G$ with respect to $(T,T_0,\alp)$,
we choose for every $(\bar G,\bar\varphi)\in \Sig$  a central extension
\begin{equation}
\label{omega}
\omega:\qquad 0\to \dbZ/p\to B\to\bar G\to 1
\end{equation}
corresponding to $\alp_{\bar G}(\bar\varphi)\in H^2(\bar G)$.
Note that $B$ depends only on $\alp_{\bar G}(\bar\varphi)$.
Let $\calL(G)$ be the class of all (isomorphism classes of) finite groups $B$ arising in this way.

\begin{examples}
\label{examples of L(G)}
\rm
Suppose that $G$ has Galois relation type.

\medskip

(1)\quad
Let $(T,T_0,\alp)$ be as in Example \ref{examples of cohomological duality triples}(1) and
let $\Sig$ be the special set for $G$ as in
Example \ref{examples of special sets}(1).
Consider $(\bar G,\bar\psi\tensor\bar\psi')\in\Sig$.
Thus $\bar\psi,\bar\psi'\in H^1(\bar G)$, $\bar\psi\cup\bar\psi'\neq0$, and $\Ker(\bar\psi)\cap\Ker(\bar\psi')=\{1\}$.
Let
\[
\omega:\qquad 0\to \dbZ/p\to B\to\bar G\to 1
\]
be the central extension corresponding to $\bar\psi\cup\bar\psi'$.

When $p\neq2$, $\bar\psi,\bar\psi'$ are $\dbF_p$-linearly independent, $\bar G\isom(\dbZ/p)^2$ and $B\isom H_{p^3}$
\cite{EfratMinac11}*{Prop.\ 9.1(f)}.
Hence $\calL(G)=\{H_{p^3}\}$.

Next let $p=2$.
When $\bar\psi=\bar\psi'$ we have $\bar G\isom\dbZ/2$ and $B\isom\dbZ/4$ \cite{EfratMinac11}*{Prop.\ 9.1(c)}.
Otherwise $\bar\psi,\bar\psi'$ are $\dbF_p$-linearly independent,
$\bar G\isom(\dbZ/2)^2$ and $B\isom D_4$ \cite{EfratMinac11}*{Prop.\ 9.1(e)}.
We conclude that $\calL(G)=\{\dbZ/4,D_4\}$.

\medskip

(2)\quad
Let $(T,T_0,\alp)$ be as in Example \ref{examples of cohomological duality triples}(2)
and let $\Sig$ be the special set for $G$ as in Example \ref{examples of special sets}(2).
Consider $(\bar G,-\bar\psi\oplus(\bar\psi\tensor\bar\psi'))\in\Sig$.
Thus $\bar\psi,\bar\psi'\in H^1(\bar G)$, $\bar\psi\neq0$, and $\Ker(\bar\psi)\cap\Ker(\bar\psi')=\{1\}$.
Let
\[
\omega:\qquad 0\to \dbZ/p\to B\to\bar G\to 1
\]
be the central extension corresponding to $-\beta_{\bar G}(\bar\psi)+\bar\psi\cup\bar\psi'$.

When $p\neq2$ and $\bar\psi,\bar\psi'$ are $\dbF_p$-linearly independent,
$\bar G\isom(\dbZ/p)^2$ and $B\isom M_{p^3}$ \cite{EfratMinac11}*{Prop.\ 9.4}.
When $p\neq2$ and $\bar\psi,\bar\psi'$ are $\dbF_p$-linearly dependent, $\bar G\isom\dbZ/p$
and $B\isom  \dbZ/p^2$ \cite{EfratMinac11}*{Cor.\ 9.3}.

\medskip

(3)\quad
Let $(T,T_0,\alp)$  be as in Example \ref{examples of cohomological duality triples}(3), and take
$\Sig$ as in Example \ref{examples of special sets}(3).
By \cite{EfratMinac11}*{Prop.\ 9.2}, the central extension corresponding to $(\bar G,\bar\psi)\in\Sig$ is
\[
0\to\dbZ/p\to\dbZ/p^2\to\bar G\ (\isom\dbZ/p)\to 1,
\]
so $\calL(G)=\{\dbZ/p^2\}$.
\end{examples}

\begin{thm}
\label{fourth main thm}
Suppose that $\Sig$ is a special set for the profinite group $G$ with respect to the cohomological duality triple $(T,T_0,A)$.
Then
\[
T_0(G)=T(G)\cap\bigcap\{M\trianglelefteq  G\ |\ G/M\in\calL(G)\}.
\]
\end{thm}
\begin{proof}
Let $(\bar G,\bar\varphi)\in\Sig$ and $\omega$ a central extension as in (\ref{omega}).
Since $\bar G$ is a quotient of $G^{[2]}$ it is also a quotient of  $G/T(G)$.
Let $\pr\colon B\times_{\bar G}(G/T(G))\to B$ be the projection from the fibred product.
The central extension
\[
\hat\omega:\qquad  0\to \dbZ/p\to B\times_{\bar G}(G/T(G))\to G/T(G)\to 1
\]
then corresponds to $\inf_{G/T(G)}(\alp_{\bar G}(\bar\varphi))$
\cite{EfratMinac11}*{Remark 6.1}.

Next let $A_0$ be the set of all elements $\alp_{G/T(G)}(\inf_{G/T(G)}(\bar\varphi))$, where $(\bar G,\bar\varphi)\in\Sig$.
Thus $A_0$ generates the kernel of $\inf\colon A(G/T(G))\to H^2(G)$.
Consider homomorphisms $\Psi\colon G\to B$, $\hat\Psi\colon G\to B\times_{\bar G}(G/T(G))$ as in the following diagram.
For every $\Psi$ making the lower triangle commutative
there is a unique $\hat\Psi$ making the upper triangle commutative with $\Psi=\pr\circ\hat\Psi$.
\[
\xymatrix{
&&&&G\ar[d]\ar[ld]_{\hat\Psi}\ar[ldd]^(.7){\Psi}\\
\hat\omega:&0\ar[r]&\dbZ/p\ar[r]\ar@{=}[d] & B\times_{\bar G}(G/T(G))\ar[r]\ar[d]_{\pr} & G/T(G)\ar[r]\ar[d]& 1\\
\omega:&0\ar[r]&\dbZ/p\ar[r]& B\ar[r]^{\pi}& \bar G\ar[r] & 1.
}
\]
Note that $\Ker(\hat\Psi)=T(G)\cap\Ker(\Psi)$.
Furthermore, if $\pi$ maps a proper subgroup $B_0$ of $B$ onto $\bar G$, then $\pi|_{B_0}\colon B_0\to\bar G$ is an isomorphism.
Since $\omega$ is non-split, $\Psi$ is therefore surjective.
Now by condition (f) of Proposition  \ref{equivelence for dualities},
$T_0(G)=\bigcap\Ker(\hat\Psi)=T(G)\cap \bigcap\Ker(\Psi)$ for all $\Psi$ as above ($T_0(G)=T(G)$ when there is no such $\Psi$).
The kernels $\Ker(\Psi)$ are just the normal subgroups $M$ of $G$ such that $G/M=B\in\calL(G)$.
\end{proof}

We now apply Theorem \ref{fourth main thm} to Examples \ref{examples of cohomological duality triples} to derive
concrete presentations of the canonical subgroups discussed so far as intersections.
The first example below contains in particular the \textsl{proof of Theorem D}.

\begin{examples}
\label{tyty}
\rm
Suppose that $G$ has Galois relation type.

\medskip

(1)\quad
Let $(T,T_0,\alp)$ be as in Example \ref{examples of cohomological duality triples}(1) and
let $\Sig$ be the special set for $G$ as in Example \ref{examples of special sets}(1).
By Example \ref{examples of L(G)}(1), $\calL(G)=\{H_{p^3}\}$ when $p\neq2$ and $\calL(G)=\{\dbZ/4,D_4\}$ when $p=2$.
In the first case we obtain from Theorem \ref{fourth main thm} that
\[
\begin{split}
G_{(3)}&=G^{(2)}\cap\bigcap\{M\trianglelefteq G\ |\ G/M\isom H_{p^3}\} \\
&=\bigcap\{M\trianglelefteq G\ |\ G/M\isom1,\dbZ/p,H_{p^3}\}.
\end{split}
\]
In view of Theorem \ref{absolute Galois groups}, this gives Theorem D.

In the second case  $G_{(3)}=G^{(3)}$ (Remark \ref{remark on the def of G3}(1)), and we obtain that
\[
G_{(3)}=\bigcap\{M\trianglelefteq G\ |\ G/M\isom1,\dbZ/2,\dbZ/4,D_4\}.
\]
This last fact was proved by Villegas \cite{Villegas88}, and Min\'a\v c and Spira \cite{MinacSpira96}*{Cor.\ 2.18}
when $G$ is an absolute Galois group of a field, and in \cite{EfratMinac11}*{Cor.\ 11.3 and Prop.\ 3.2}
for general profinite groups of Galois relation type.
Moreover, $\dbZ/2$ can be omitted from this list if $G\not\isom\dbZ/2$ \cite{EfratMinac11}*{Cor.\ 11.4}.

\medskip

(2)\quad
Let $(T,T_0,\alp)$ be as in Example \ref{examples of cohomological duality triples}(2)
and let $\Sig$ be the special set for $G$ as in Example \ref{examples of special sets}(2).
We may assume that $p\neq2$, since otherwise $G^{(3)}=G_{(3)}$, and this subgroup was described in (1).
Now by Example \ref{examples of L(G)}(2), $\calL(G)=\{M_{p^3},\dbZ/p^2\}$.
By \cite{EfratMinac11}*{Prop. 10.2}, every epimorphism $G\to \dbZ/p$ breaks via $\dbZ/p^2$ or via $M_{p^3}$.
Consequently, Theorem \ref{fourth main thm} gives
\[
G^{(3)}=\bigcap\{M\trianglelefteq G\ |\ G/M\isom1,\dbZ/p^2,M_{p^3}\}.
\]
This result was earlier proved in \cite{EfratMinac11}.

\medskip

(3)\quad
Let $(T,T_0,\alp)$  be as in Example \ref{examples of cohomological duality triples}(3), and take
$\Sig$ as in Example \ref{examples of special sets}(3).
By Example \ref{examples of L(G)}(2), $\calL(G)=\{\dbZ/p^2\}$.
We get the equality (already noted in Example \ref{examples of duality}(3))
\[
G^{p^2}[G,G]=\bigcap\{M\trianglelefteq G \ |\ G/M\isom1,\dbZ/p,\dbZ/p^2\}.
\]
In fact, since a discrete abelian group of finite exponent
is a direct sum of cyclic groups \cite{Kaplansky69}*{Th.\ 6}, we get using Pontrjagin duality that
\[
G^{p^n}[G,G]=
\bigcap\{M\trianglelefteq G \ |\ G/M\isom\dbZ/p^j,\ j=0,1\nek n\}.
\]
\end{examples}

\section{Duality in free pro-$p$ groups}
\label{section on free groups}
In this section we prove the duality mentioned in Example \ref{examples of duality}(1).
First we study the pairing in Proposition \ref{generalized basic duality}(b) when $G=S$ satisfies $H^2(S)=0$ and when
$T=S^{(2)}$ and $T_0=S^{(3)}$.
Given $\sig\in S$, we write $\bar\sig$ for its image in $S^{[2]}$.
The next two lemmas extend computations in \cite{Labute66}*{\S2.3}, \cite{Koch02}*{\S7.8}
and \cite{NeukirchSchmidtWingberg}*{3.9.13}.
Let $\delta=1,2$ be as in (\ref{delta}).

\begin{lem}
\label{lemma on cup products}
Let $\chi,\chi'\in H^1(S^{[2]})$ and $\sig,\sig'\in S$.
\begin{enumerate}
\item[(a)]
If there is a homomorphism $h\colon S\to \dbZ/q$ with $h(\sig)=\bar1$, then
$\langle \sig^qS^{(3)}, \chi\cup\chi'\rangle=(q/\del)\chi(\bar\sig)\chi'(\bar\sig)$ $(=0$ for $q$ odd).
\item[(b)]
$\langle[\sig,\sig']S^{(3)},\chi\cup\chi'\rangle
=\chi(\bar\sig')\chi'(\bar\sig)-\chi(\bar\sig)\chi'(\bar\sig')$.
\end{enumerate}
\end{lem}
\begin{proof}
The cohomology class $\inf_S(\chi\cup\chi')$ in $H^2(S)$ is represented by the $2$-cocycle
$c(\sig,\tau)=\chi(\bar\sig)\chi'(\bar\tau)$.
Since $H^2(S)=0$, there exists an inhomogenous $1$-cochain $u\colon S\to\dbZ/q$ with $\partial u=c$.
Thus
\begin{equation}
\label{cochain and chi 2}
\chi(\bar\sig)\chi'(\bar\tau)=u(\sig)+u(\tau)-u(\sig\tau)
\end{equation}
for all $\sig,\tau\in S$.
In particular, for $\sig\in S$ and $\tau\in S^{(2)}$ we have $u(\sig\tau)=u(\sig)+u(\tau)=u(\tau\sig)$.
It follows that for $\tau\in S^{(2)}$ one has
\[
\begin{split}
u(\tau)&=u(\sig\inv)+u(\sig\tau)-\chi(\bar\sig\inv)\chi'(\bar\sig)\\
&=u(\sig\inv)+u(\tau\sig)-\chi(\bar\sig\inv)\chi'(\bar\sig)
=u(\sig\inv\tau\sig).
\end{split}
\]
Thus the restriction $v$ of $u$ to $S^{(2)}$ belongs to $H^1(S^{(2)})^ S$.
By the definition of the transgression \cite{NeukirchSchmidtWingberg}*{Prop.\ 1.6.6}, $\trg_{S^{[2]}}(v)=\chi\cup\chi'$.
Consequently, for every $\rho\in S^{(2)}$ we have
\[
\langle \rho S^{(3)},\chi\cup\chi'\rangle=
\langle \rho S^{(3)}, \trg_{S^{[2]}}(v)\rangle=-v(\rho)=-u(\rho).
\]

(a)\quad
When $\sig\in S^{(2)}$ both sides are zero.
So assume that $\sig\not\in S^{(2)}$.
Our assumption gives a homomorphism $h\colon S\to\dbZ/q$ with $h(\sig)=u(\sig)$.
As $\partial h=0$, we may replace $u$ by $u-h$ to assume that $u(\sig)=0$.
Using (\ref{cochain and chi 2}) we obtain inductively that
$u(\sig^i)=-\binom i2\chi(\bar\sig)\chi'(\bar\sig)$.
It remains to observe that $\binom q2\equiv q/\del\pmod q$.

\medskip

(b)\quad
Apply (\ref{cochain and chi 2}) with $\tau=1$ to obtain $u(1)=0$.
Apply it with $\tau=\sig\inv$ to further obtain $u(\sig\inv)+u(\sig)=-\chi(\bar\sig)\chi'(\bar\sig)$.
This gives
\[
\begin{split}
u((\sig'\sig)\inv)+u(\sig'\sig)&=-\chi(\bar\sig'\bar\sig)\chi'(\bar\sig'\bar\sig) \\
u(\sig\sig')&=u(\sig)+u(\sig')-\chi(\bar\sig)\chi'(\bar\sig') \\
-u(\sig'\sig)&=-u(\sig')-u(\sig)+\chi(\bar\sig')\chi'(\bar\sig) \\
\end{split}
\]
Adding these equalities we obtain
\[
u((\sig'\sig)\inv)+u(\sig\sig')
=-\chi(\bar\sig'\bar\sig)\chi'(\bar\sig'\bar\sig)+\chi(\bar\sig')\chi'(\bar\sig)-\chi(\bar\sig)\chi'(\bar\sig').
\]
Hence, by (\ref{cochain and chi 2}) again,
\[
\begin{split}
u([\sig,\sig'])&=
u((\sig'\sig)\inv)+u(\sig\sig')-\chi((\bar\sig\bar\sig')\inv)\chi'(\bar\sig\bar\sig') \\
&=\chi(\bar\sig')\chi'(\bar\sig)-\chi(\bar\sig)\chi'(\bar\sig').
\end{split}
\]
as desired.
\end{proof}

\begin{lem}
\label{lemma on Bocksteins}
Let $\chi\in H^1(S^{[2]})$ and $\sig,\sig'\in S$.
\begin{enumerate}
\item[(a)]
If $\chi(\bar\sig)=\bar0$, then $\langle \sig^qS^{(3)}, \beta_{S^{[2]}}(\chi)\rangle=\bar0$.
\item[(b)]
If $\chi(\bar\sig)=\bar1$, then $\langle \sig^qS^{(3)}, \beta_{S^{[2]}}(\chi)\rangle=\bar1$.
\item[(c)]
$\langle [\sig,\sig']S^{(3)}, \beta_{S^{[2]}}(\chi)\rangle=\bar0$.
\end{enumerate}
\end{lem}
\begin{proof}
Define a section $\iota$ of the projection $\dbZ/q^2\to\dbZ/q$ by $\iota(i+q\dbZ)=i+\dbZ/q^2$ for $0\leq i\leq q-1$.
Let $\tilde\chi=\iota\circ\chi\in\Hom(S^{[2]},\dbZ/q^2)$.
Then the cohomology class of $\beta_{S^{[2]}}(\chi)$ in $H^2(S^{[2]})$ is represented by the
$2$-cocycle
\[
c(\bar\sig,\bar\tau)=\dfrac1q\bigl(\tilde\chi(\bar\sig)+\tilde\chi(\bar\tau)-\tilde\chi(\bar\sig\bar\tau)\bigr).
\]
Inflating to $H^2(S)=0$, we obtain an inhomogenous $1$-cochain $u\colon S\to\dbZ/q$ with $\partial u=c$.
Thus
\begin{equation}
\label{cochain and chi 1}
u(\sig)+u(\tau)-u(\sig\tau)=\dfrac1q\bigl(\tilde\chi(\bar\sig)+\tilde\chi(\bar\tau)-\tilde\chi(\bar\sig\bar\tau)\bigr)
\end{equation}
for all $\sig,\tau\in S$.
Since $S^{[2]}$ is abelian, this implies that $u(\sig\tau)=u(\tau\sig)$, whence $u(\sig\inv\tau\sig)=u(\tau)$, for all
$\sig,\tau\in S$.
It follows that the restriction $v$ of $u$ to $S^{(2)}$ belongs to $H^1(S^{(2)})^ S$.
By the definition of the transgression,
$\trg_{S^{[2]}}(v)$ is represented by the $2$-cocycle $\partial u=c$.
Hence $\beta_{S^{[2]}}(\chi)=\trg_{S^{[2]}}(v)$.
Consequently, for every $\rho\in S^{(2)}$ we have
\[
\langle \rho S^{(3)},\beta_{S^{[2]}}(\chi)\rangle=
\langle \rho S^{(3)}, \trg_{S^{[2]}}(v)\rangle=-v(\rho)=-u(\rho).
\]

(a)\quad
If $\chi(\sig)=\bar0$, then $\tilde\chi(\sig^i)=0$ for every $i\geq0$.
It follows inductively from (\ref{cochain and chi 1}) that $u(\sig^i)=iu(\sig)$ for $i\geq0$.
Thus $u(\sig^q)=\bar0$, as required.

\medskip

(b)\quad
If $\chi(\sig)=\bar1$, then $\tilde\chi(\sig^i)=\bar i$ for $0\leq i\leq q-1$,
while $\tilde\chi(\sig^q)=\bar0$.
It follows from (\ref{cochain and chi 1}) by induction that $u(\sig^i)=iu(\sig)$ for $0\leq i\leq q-1$, and
$u(\sig^q)=qu(\sig)-\bar1=-\bar1$.

\medskip

(c)\quad
As $(\sig,\tau)\mapsto u(\sig\tau)$ is symmetric, $u([\sig,\sig'])=\bar0$ for all $\sig,\sig'\in S$.
\end{proof}

Next we also assume that $S^{[2]}\isom(\dbZ/q)^I$ for some index set $I$, e.g., $S$ is a free profinite (or pro-$p$ group) group.
Fix a linear order $<$ on $I$.
Choose a basis $\bar\sig_i$, $i\in I$, of $S^{[2]}$, and lift it to elements $\sig_i$, $i\in I$, of $S$.
Let $\chi_i$, $i\in I$, be the $\dbZ/q$-basis of $H^1(S^{[2]})$ dual to $\sig_i S^{(2)}$, $i\in I$.

\begin{prop}
\label{cup products}
\begin{enumerate}
\item[(a)]
Every $\sig\in S^{(2)}$ can be uniquely written as
\[
\sig=\prod_i\sig_i^{a_iq}\prod_{i<j}[\sig_i,\sig_j]^{b_{ij}}\pmod {S^{(3)}}
\]
for some some $a_i,b_{ij}\in\dbZ/q$.
\item[(b)]
The lists
\begin{enumerate}
\item[(i)]
$\sig_i^qS^{(3)}$, $i\in I$, and $[\sig_i,\sig_j]S^{(3)}$, $i,j\in I$, $i<j$;
\item[(ii)]
$\beta_{S^{[2]}}(\chi_i)$, $i\in I$, and  $\chi_i\cup\chi_j$, $i,j\in I$, $i<j$;
\end{enumerate}
form dual bases of $S^{(2)}/S^{(3)}$ and $H^2(S^{[2]})$, respectively, with respect to $\langle\cdot,\cdot\rangle$.
\end{enumerate}
\end{prop}
\begin{proof}
(a)\quad
Use \cite{NeukirchSchmidtWingberg}*{Prop.\ 3.9.13(i)} and a standard limit argument.

\medskip

(b)\quad
By the definition of $S^{(2)}$, the list (i) generates $S^{(2)}/S^{(3)}$.
By Lemma \ref{lemma on cup products} and Lemma \ref{lemma on Bocksteins}, the two lists are dual.
Hence they form bases.
\end{proof}

\begin{prop}
\label{identity}
For $\chi\in H^1(S^{[2]})$ of order $q$ one has $\chi\cup\chi=(q/\del)\beta_{S^{[2]}}(\chi)$ $(=0$ if $p>2)$.
\end{prop}
\begin{proof}
We may assume that $\chi=\chi_k$ for some $k\in I$.
For $i\in I$ Lemma \ref{lemma on cup products}(a) and Lemma \ref{lemma on Bocksteins}(a)(b)
give
\[
\langle \sig_i^qS^{(3)},\chi_k\cup\chi_k\rangle=(q/\del)\chi_k(\sig_i)
=\langle\sig_i^qS^{(3)},(q/\del)\beta_{S^{[2]}}(\chi_k)\rangle.
\]
For $i,j\in I$, $i<j$, Lemma \ref{lemma on cup products}(b) and Lemma \ref{lemma on Bocksteins}(c) give
\[
\langle [\sig_i,\sig_j]S^{(3)},\chi_k\cup\chi_k\rangle=0
=\langle[\sig_i,\sig_j]S^{(3)},(q/\del)\beta_{S^{[2]}}(\chi_k)\rangle.
\]
The assertion now follows by duality from Proposition \ref{cup products}(b).
\end{proof}

From this and Proposition \ref{cup products} we deduce

\begin{cor}
\label{decomposable part of H2(S2)}
When $p>2$ (resp., $p=2$), the elements $\chi_i\cup\chi_j$, $i<j$
(resp.\  $i\leq j$) form a basis of $H^2(S^{[2]})_\dec$.
\end{cor}

To this end let $S$ be again a profinite group such that $H^2(S)=0$ and $S^{[2]}\isom(\dbZ/q)^I$.
Let $\sig_i$, $\chi_i$, $i\in I$, be bases as in the previous section.
Write $S^{[2]}=\prod_{i\in I}C_i$, where $C_i=\langle\sig_i S^{(2)}\rangle\isom\dbZ/q$.
Then $\res_{C_i}(\chi_i)$ generates $H^1(C_i)\isom\dbZ/q$, and $\beta_{C_i}(\res_{C_i}(\chi_i))$
generates $H^2(C_i)\isom\dbZ/q$, e.g.\  by Corollary \ref{decomposable part of H2(S2)}.

\begin{lem}
\label{coefficients}
Let $\alp\in H^2(S^{[2]})$ and let $d_i$, $i\in I$, and $d_{ij}$, $i<j$, $i,j\in I$,  be the unique elements
of $\dbZ/q$ such that
\begin{equation}
\label{ttt}
\alp=\sum_id_i\beta_{S^{[2]}}(\chi_i)+\sum_{i<j}d_{ij}\cdot\chi_i\cup\chi_j.
\end{equation}
Then for every  $k\in I$, one has $\res_{C_k}(\alp)=d_k\beta_{C_k}(\res_{C_k}(\chi_k))$.
\end{lem}
\begin{proof}
One has $\res_{C_k}(\chi_i)=0$ for $i\neq k$.
Therefore
\[
\res_{C_k}(\beta_{S^{[2]}}(\chi_i))=\beta_{C_k}(\res_{C_k}(\chi_i))=0
\]
for $i\neq k$, as well as
\[
\res_{C_k}(\chi_i\cup\chi_j)=\res_{C_k}(\chi_i)\cup\res_{C_k}(\chi_j)=0
\]
for all $i<j$, whence the desired equality.
\end{proof}

We deduce the following local-global principle for groups of the form $(\dbZ/q)^I$.

\begin{cor}
\label{local global}
There is an exact sequence
\[
0\to H^2(S^{[2]})_\dec\hookrightarrow H^2(S^{[2]})\xrightarrow{\prod\delta\res_C} \textstyle\prod_C H^2(C),
\]
where $C$ ranges over all cyclic subgroups of  $S^{[2]}$ of order $q$.
\end{cor}
\begin{proof}
Let $\alp\in H^2(S^{[2]})$ and express it as in (\ref{ttt}).
By Corollary \ref{decomposable part of H2(S2)}, $\alp\in H^2(S^{[2]})_\dec$ if and only if $\del d_i\equiv 0\pmod q$
for every $i\in I$.
Since $\beta_{C_i}(\res_{C_i}(\chi_i))$ generates $H^2(C_i)$, Lemma \ref{coefficients}
shows that this is equivalent to $\del\res_{C_i}(\alp)=0$.
It remains to note that every cyclic subgroup $C$ of $S^{[2]}$ of order $q$
occurs as $C_k=\langle\sig_kS^{(2)}\rangle$ for some choice of $\sig_i$, $\psi_i$.
\end{proof}

Now let $G$ be a profinite group.
Given a subgroup $C$ of $G^{[2]}$, take a normal subgroup $M$ of $G$ containing $G^{(2)}$
with $C=M/G^{(2)}$.
Then $H^1(G^{(2)})^G\leq H^1(G^{(2)})^M$, and by the functoriality of the transgression,
there is a commutative square
\begin{equation}
\label{sss}
\xymatrix{
H^1(G^{(2)})^G \ar[r]^{\trg_{G^{[2]}}}  \ar@{_{(}->}[d] & H^2(G^{[2]})\ar[d]^{\res_C} \\
H^1(G^{(2)})^M \ar[r]^{\trg_C} & H^2(C) .
}
\end{equation}
This makes condition (c) of the following proposition meaningful:

\begin{prop}
\label{kkk}
Let $G$ be a profinite group with $G^{[2]}\isom(\dbZ/q)^I$ for some index set $I$.
The following conditions on $\psi\in H^1(G^{(2)})^G$ are equivalent:
\begin{enumerate}
\item[(a)]
$\trg_{G^{[2]}}(\psi)\in H^2(G^{[2]})_\dec$;
\item[(b)]
$\del\trg_C(\psi)=0$ for every cyclic subgroup $C$ of $G^{[2]}$ of order $q$;
\item[(c)]
for every normal subgroup $M$ of $G$ containing $G^{(2)}$ with $M/G^{(2)}\isom\dbZ/q$
there exists $\hat\psi\in H^1(M)$ with $\del\psi=\res_{G^{(2)}}(\hat\psi)$;
\item[(d)]
$\psi(G_{(3)})=0$.
\end{enumerate}
\end{prop}
\begin{proof}
(a)$\Leftrightarrow$(b):  \quad
Since $G^{[2]}\isom S^{[2]}$ for some free pro-$p$ group $S$,
this follows from Corollary \ref{local global} and (\ref{sss}).

\medskip

(b)$\Leftrightarrow$(c):\quad
Use the 5-term sequence associated with $C=M/G^{(2)}$.

\medskip

(c)$\Leftrightarrow$(d):\quad
We may assume that $G\neq\{1\}$.
Then $G$ is the union of its subgroups $M$ such that $G^{(2)}\leq M$ and $M/G^{(2)}\isom\dbZ/q$.
Then $M^q\leq G^{(2)}$ and there is a split extension
\[
1\to G^{(2)}/M^q\to M/M^q\to\dbZ/q\to0.
\]

If there is $\hat\psi\in H^1(M)$ with $\del\psi=\res_{G^{(2)}}(\hat\psi)$,
then $(\del\psi)(M^q)=\hat\psi(M^q)=\{0\}$.
Conversely, if $\del\psi$ vanishes on $M^q$, then it induces a homomorphism $\overline{\del\psi}\in H^1(G^{(2)}/M^q)^M$.
Since the above extension splits, $\overline{\del\psi}$ extends to a homomorphism $M/M^q\to\dbZ/q$,
so there is $\hat\psi\in H^1(M)$ as above.

Consequently, (c) is equivalent to $\psi(G^{\del q})=\{1\}$.
But as $\psi\in H^1(G^{(2)})^G$, one always has $\psi([G^{(2)},G])=\{1\}$, so $\psi(G_{(3)})=\psi(G^{\del q})$.
\end{proof}

The equivalence of (a) and (d) implies

\begin{cor}
\label{mmm}
Let $G$ be a profinite group with $G^{[2]}\isom(\dbZ/q)^I$ for some index set $I$.
Then  $H^2(G^{[2]})_\dec$ is dual to $(G^{(2)},G_{(3)})$.
\end{cor}

\end{document}